%% file: main.tex
\documentclass{siamart220329}

\usepackage{siampaper}

\usepackage[normalem]{ulem}
\usepackage{comment}

\usepackage{backref}

\renewcommand*{\backref}[1]{}% for backref < 1.33 necessary
\renewcommand*{\backrefalt}[4]{%
\ifcase #1 %
No citations.%
\or
\marginpar{\tiny cited on p. #2}%
\else
\marginpar{\tiny cited on pp. #2}%
\fi
}

\makeatletter
\AtBeginDocument{%
  \def\refstepcounter@optarg[#1]#2{%
    \cref@old@refstepcounter{#2}%
    \hyper@makecurrent{#2}%
\Hy@raisedlink{\hyper@anchorstart{\@currentHref}\hyper@anchorend}%
    \cref@constructprefix{#2}{\cref@result}%
    \if\relax\detokenize{#1}\relax
      \def\@tempa{#2}%
    \else
      \@ifundefined{cref@#1@alias}%
        {\def\@tempa{#1}}%
        {\def\@tempa{\csname cref@#1@alias\endcsname}}%
    \fi
    \protected@edef\cref@currentlabel{%
      [\@tempa][\arabic{#2}][\cref@result]%
      \csname p@#2\endcsname\csname the#2\endcsname}}}%
\makeatother

\usepackage{xstring}

\makeatletter
\let\siamoriginalLabel\label
\renewcommand{\label}[1]{%
  \Hy@raisedlink{\hyper@anchorstart{ref.#1}\hyper@anchorend}%
  \def\@currentHref{ref.#1}%
  \IfBeginWith{#1}{sec:}{\protected@edef\cref@currentlabel{[section][\arabic{section}][]\@currentlabel}}{}%
  \IfBeginWith{#1}{Sec:}{\protected@edef\cref@currentlabel{[section][\arabic{section}][]\@currentlabel}}{}%
  \siamoriginalLabel{#1}%
}
\makeatother

\crefname{algorithm}{Algorithm}{Algorithms}
\Crefname{algorithm}{Algorithm}{Algorithms}
\crefname{}{section}{sections}
\Crefname{}{Section}{Sections}

\input{ex_shared}
\DeclareMathOperator{\diag}{diag}

\ifpdf
\hypersetup{
  pdftitle={Deflation-Based Spectral Preconditioning for the Conjugate Gradient Method on Perturbed Systems},
  pdfauthor={Y. Diouane, S. Gurol, O. Mouhtal, and D. Orban}
}
\fi

\begin{document}

\maketitle
\thispagestyle{firstpage}
\pagestyle{myheadings}

% REQUIRED
\begin{abstract}
We consider sequences of large symmetric positive-definite linear systems in which spectral information from a reference matrix $A$ is reused to accelerate the solution of a nearby system with perturbed matrix $A(t)=A+tE$, where $E$ is symmetric, $\|E\|_F=1$, and $t\geq0$ is the perturbation size. We compare deflated conjugate gradient (deflated CG) with preconditioned CG (PCG) equipped with a scaled spectral preconditioner constructed from the same spectral information of $A$. First, we derive an energy-norm error characterization of deflated CG. We then use perturbation analysis to show that, after one additional PCG iteration, the squared PCG error is bounded, up to terms of order $t^2$, by an explicit multiple of the squared deflated CG error. This result clarifies the role of the scaling parameter in the preconditioner. Experiments show that the scaled spectral preconditioner recovers much of the convergence improvement provided by deflation while requiring fewer products with $A(t)$.

\end{abstract}


% REQUIRED

\section{Introduction}

We consider sequences of large symmetric positive-definite (SPD) linear systems solved by the conjugate gradient (CG) method~\citep{Hestenes1952MethodsOC}. Such sequences arise in many applications, such as variational data assimilation~\citep{DataAssimilaton}, where consecutive systems are often closely related. Information generated while solving one system can then be recycled to accelerate subsequent solves; see, for example,~\citep{parks2006recycling,Stathopoulos,gosselet2013total,TshiGratWeavSart08}.

We consider a reference matrix $A$ together with a nearby matrix $A(t)=A+tE$,
where $E$ is symmetric, $\|E\|_F=1$, and $t$ is sufficiently small for $A(t)$ to remain positive definite. Our goal is to reuse spectral information from $A$ to accelerate CG applied to the perturbed system $A(t)x=b$.
This is particularly attractive in large-scale matrix-free settings, where constructing a preconditioner may be prohibitively expensive and only a limited number of matrix--vector products can be afforded.

One important approach for exploiting recycled information is deflated CG. Given a deflation subspace $\mathcal{W}=\operatorname{Range}(W)$, deflated CG splits the solution into a component in $\mathcal{W}$ and a complementary component computed by CG on a projected system~\citep{nicolaides1987deflation,DeflatedCG,frank2001construction}. \citet{Tang2009} mention that, for any deflation subspace, the effective condition number of the deflated system does not exceed that of the original system. A widely used choice for the deflation subspace is a set of approximate eigenvectors obtained from previous Krylov iterations through Ritz or harmonic Ritz projections~\citep{Stathopoulos,DeflatedCG,parks2006recycling,venkovic2020comparative}. In the nonsymmetric setting, SVD-based strategies have also been proposed for recycling Krylov subspaces, in particular for GMRES~\citep{AlDaas2021}. Problem-dependent alternatives include coarse spaces arising from domain decomposition and multilevel methods~\citep{Tang2009}. These approaches seek to construct a low-dimensional subspace that captures the components responsible for slow convergence at moderate additional cost.

Approaches related to deflation are limited-memory preconditioners (LMPs)~\citep{LMP}, augmentation methods~\citep{Gaul}, balancing strategies~\citep{nabben2006comparison}, and low-rank updates~\citep{Bergamaschi2020}. When spectral information is  used in the preconditioner, we refer to such methods as spectral preconditioners~\citep{carpentieri2003class,SenSpectLMP,LMP}. For broader overviews of preconditioning, see~\citep{Preconditionning1,PearsonPestana2020,wathen2015preconditioning}.

The effect of using approximate spectral information has therefore received considerable attention. The sensitivity of deflated CG to perturbations of the deflation subspace is studied in~\citep{Kahl2017}, while~\citet{SenSpectLMP} analyze how errors in approximate eigenpairs affect the spectrum of the preconditioned matrix. These results characterize the influence of inexact spectral information on the projected or preconditioned spectrum.
Although deflation can significantly improve CG convergence, its construction may become costly when the coefficient matrix changes. For a deflation basis $W\in\mathbb{R}^{n\times k}$, one must form and solve a system with the matrix $W^\top A(t)W$, which requires additional products with $A(t)$. This problem generally needs to be solved accurately, since the convergence of deflated CG can be sensitive to errors in its solution~\citep{nabben2006comparison}. To mitigate this sensitivity, related approaches avoid mapping the selected eigenvalues to zero, e.g.,~\citet{erlangga2008multilevel} shift selected eigenvalues to a nonzero part of the spectrum and use the resulting preconditioner within GMRES.

In~\citep{diouane2026spectral}, the authors introduced a scaled spectral preconditioner designed to reduce the energy-norm error after a prescribed number of preconditioned CG (PCG) iterations. The preconditioner is constructed from the eigenvalues and eigenvectors of $A$ and maps the selected eigenvalues to a common value $\theta$. Experiments showed that an appropriate choice of the scaling parameter can reproduce much of the convergence acceleration obtained with deflation while requiring a lower initialization cost. In particular, this behavior was observed in a data assimilation problem when approximate eigenpairs were recycled. However, the theoretical analysis in~\citep{diouane2026spectral} assumes that the spectral information used to construct the preconditioner is exact and is associated with the matrix of the system being solved. It therefore does not explain why similar behavior persists when spectral information from a nearby matrix is reused.

Our main objective in this work is to establish this connection theoretically when spectral information obtained from $A$ is reused for the perturbed system $A(t)x=b$. We focus on eigenpairs associated with the largest eigenvalues of $A$, although the analysis extends to any selected subset of eigenpairs. We compare the energy-norm errors generated by deflated CG and by PCG equipped with the scaled spectral preconditioner. Using first-order perturbation analysis, we show that, after one additional PCG iteration, the squared energy-norm error of PCG is bounded by the squared energy-norm error of deflated CG multiplied by an explicit factor, plus higher-order terms in $t$. The resulting bound clarifies how the scaling parameter $\theta$ should be chosen relative to the nondeflated spectrum and provides a theoretical explanation for the similar convergence behavior observed numerically. We complement the theoretical analysis with numerical experiments on synthetic linear systems.

The paper is organized as follows. In~\Cref{sec:DefCG}, we review deflated CG and derive a characterization of its energy-norm error. We formulate the perturbed problem, develop the first-order analysis, establish the comparison between PCG and deflated CG, and discuss practical choices of the scaling parameter in~\Cref{sec:perturbed_system}. We report the numerical experiments in~\Cref{sec:numerical_experiments} and present our conclusions in~\Cref{sec:conclusion}.

\noindent \textbf{Notation:}
The spectral norm of a matrix $X$ is denoted $\| \|_2$, its spectral radius $\rho(X)$, and its range $\mathcal{R}(X)$. For a symmetric matrix $X \in \mathbb{R}^{n\times n}$, we denote its eigenvalues ${\lambda_1(X),\ldots,\lambda_n(X)}$ and the corresponding eigenvectors ${s_1(X),\ldots,s_n(X)}$. 
Throughout the paper, $A\in\mathbb{R}^{n\times n}$ is SPD. When there is no ambiguity, its $i$th eigenvalue and corresponding eigenvector are denoted $\lambda_i$ and $s_i$, respectively. For $k\in\{1,\ldots,n-1\}$, we define $S_k=[s_1,\ldots,s_k]\in\mathbb{R}^{n\times k}$ and $\widebar S_k=[s_{k+1},\ldots,s_n]\in\mathbb{R}^{n\times(n-k)}$. Thus, $S_k$ contains the eigenvectors associated with the $k$ largest eigenvalues of $A$, while the columns of $\widebar S_k$ span the orthogonal complement of $\mathcal{R}(S_k)$.
For $x\in\mathbb{R}^n$, the $A$-norm, also called the \emph{energy norm}, is defined by $\|x\|_A=\sqrt{x^\top A x}$. We denote $\mathbb{P}_\ell(0)$ the set of polynomials of degree at most $\ell$ with value $1$ at zero.

In what follows, all results are derived under the assumption of exact arithmetic.

\section{Deflated conjugate gradient}
\label{sec:DefCG}
Let $W\in \mathbb{R}^{n \times k}$ have full column rank. Deflated CG~\citep{DeflatedCG,frank2001construction} decomposes the solution of $Ax=b$ into a
component in the deflation subspace $\mathcal R(W)$ and a complementary component computed by CG on a projected system.
Let $V=AW$ and define the projection $
P \;=\; I_n - V \bigl(W^{\top} V\bigr)^{-1} W^{\top}$.
Deflated CG then generates iterates of the form~\citep{DeflatedCG}
\begin{equation}\label{eq:defcg_iterates_gecase}
x_\ell^{\text{D}} \;=\; W \bigl(W^{\top} V\bigr)^{-1} W^\top b \;+\; P^\top z_\ell,
\end{equation}
where \( z_\ell \) are the iterates of CG applied to the projected system
\begin{equation}\label{eq:ProjSys}
\widehat{A} z \;=\; \widehat{b},
\end{equation}
with \( \widehat{A} = P A \), \( \widehat{b} = P b \), and arbitrary starting vector \( z_0 \).  
The projected system~\eqref{eq:ProjSys} can be solved by CG~\citep{kammerer1972convergence}, since \( \widehat{A} \) is symmetric positive semidefinite, satisfies
\(\dim(\mathcal{R}(\widehat{A})) = n - k\)~\citep{DeflatedCG}, and \(\widehat{b} \in \mathcal{R}(\widehat{A})\). Consider the eigen-decomposition
\begin{equation*}
\widehat{A} \;=\; [Q_k\; \widebar Q_k]
\begin{bmatrix}
0 & 0 \\ 0 & \widebar{\Delta}_k
\end{bmatrix}
[Q_k\; \widebar Q_k]^\top,
\end{equation*}
where \(Q_k\in\mathbb R^{n\times k}\) spans the nullspace of \(\widehat A\), \( \widebar{\Delta}_k \in \mathbb{R}^{(n-k) \times (n-k)} \) is diagonal of nonzero eigenvalues of \(\widehat{A}\), and \( \widebar Q_k \in \mathbb{R}^{n \times (n-k)} \)  the corresponding eigenvectors. In the next lemma, we show that solving~\eqref{eq:ProjSys} is equivalent to solving the reduced system
\begin{equation}\label{eq:small_system}
\widebar{\Delta}_k \widetilde{y} \;=\; \widetilde{b},\qquad \mbox{where} \qquad \widetilde{b} = \widebar Q_k^\top \widehat{b}.\end{equation}
\begin{lemma}
\label{lemma:CG_xproj_yproj}
Let $\{z_\ell\}$ be the iterates generated by CG applied to~\eqref{eq:ProjSys} with an initial vector $z_0$, and let $\{\widetilde{y}_\ell\}$ be the iterates generated by CG applied to~\eqref{eq:small_system} with initial vector $\widetilde{y}_0=\widebar Q_k^{\top}z_0$. Let \(\widetilde{r}_\ell\) and \(\widetilde{p}_{\ell}\) be the residuals and search directions when solving~\eqref{eq:small_system} with CG, and \(\widehat{r}_\ell\) and \(\widehat{p}_\ell\) be the residuals and search directions when solving~\eqref{eq:ProjSys} with CG. Then, for every $\ell\in\{0,\ldots,n-k\}$ such that $\widetilde{y}_\ell$ exists, 
\begin{align}
\label{CG-equivalence1}
    \widetilde{r}_\ell & = \widebar Q_k^\top \widehat{r}_\ell \quad \text{and} \quad \widehat{r}_\ell \in \mathcal{R}(\widehat{A}) = \mathcal{R}(\widebar Q_k), \\
    \label{CG-equivalence2}
    \widetilde{p}_{\ell} & = \widebar Q_k^\top \widehat{p}_\ell  \quad \text{and} \quad \widehat{p}_\ell \in \mathcal{R}(\widehat{A}),  \\
    \label{CG-equivalence3}
    \widetilde{y}_\ell & = \widebar Q_k^\top z_\ell, \\
    \label{CG-equivalence4}
    Q_k Q_k^\top z_\ell & = Q_k Q_k^\top z_0.
\end{align}
\end{lemma}

\begin{proof}
We proceed by induction. For $\ell=0$,
\begin{align*}
    \widetilde r_0= \widebar Q_k^\top \widehat{b} - \widebar{\Delta}_k \widetilde{y}_0
    = \widebar Q_k^\top (\widehat{b} - \widebar Q_k \widebar{\Delta}_k  \widebar Q_k^\top z_0) = \widebar Q_k^\top(\widehat{b} - \widehat{A}z_0) = \widebar Q_k^\top  \widehat{r}_0.
\end{align*}
Because the projected system is consistent, $\widehat b\in\mathcal R(\widehat A)$; hence $\widehat r_0\in\mathcal R(\widehat A)$. Furthermore, $\widetilde{p}_0 =\widetilde{r}_0$ and $\widehat p_0=\widehat r_0$, so $\widetilde{p}_0=\widebar Q_k^\top\widehat p_0$. The identities~\eqref{CG-equivalence3} and~\eqref{CG-equivalence4} hold at $\ell=0$ by definition.
Assume now that~\eqref{CG-equivalence1}--\eqref{CG-equivalence4} hold at iteration $\ell<n-k$ and CG has not terminated. 
The CG update for the linear system~\eqref{eq:small_system} gives

\begin{align*}
    \widetilde{y}_{\ell+1} &=\widetilde{y}_{\ell} + \dfrac{\widetilde{r}_\ell^\top \widetilde{r}_\ell}{\widetilde{p}_{\ell}^\top \widebar{\Delta}_k \widetilde{p}_{\ell}} \widetilde{p}_{\ell} = \widebar Q_k^\top z_\ell + \dfrac{{\widehat{r}_\ell}^\top \widebar Q_k\widebar Q_k^\top \widehat{r}_\ell}{\widehat{p}_\ell^\top \widebar Q_k \widebar{\Delta}_k \widebar Q_k^\top \widehat{p}_\ell} \widebar Q_k^\top \widehat{p}_\ell.
\end{align*}
Since $\widehat{r}_\ell \in \mathcal{R}(\widebar Q_k)$, then $\widebar Q_k\widebar Q_k^\top \widehat{r}_\ell = \widehat{r}_\ell$. We get
$ \widetilde{y}_{\ell+1} = \widebar Q_k^\top\left( z_\ell + \dfrac{\widehat{r}_\ell^\top \widehat{r}_\ell}{\widehat{p}_\ell^\top \widehat{A} \widehat{p}_\ell} \widehat{p}_\ell\right)
    =\widebar Q_k^\top z_{\ell+1}$.
Using this identity, the reduced residual satisfies
\begin{align*}
    \widetilde{r}_{\ell+1} = \widebar Q_k^\top \widehat{b} - \widebar{\Delta}_k \widetilde{y}_{\ell+1}
    =\widebar Q_k^\top \left(\widehat{b} - \widebar Q_k\widebar{\Delta}_k \widebar Q_k^\top z_{\ell+1}\right)
    = \widebar Q_k^\top(\widehat{b} - \widehat{A}z_{\ell+1})
    =\widebar Q_k^\top \widehat{r}_{\ell+1},
\end{align*}
where $\widehat{r}_{\ell+1} \in \mathcal{R}(\widehat{A})$. For the search direction, from the CG recurrence, we get
\begin{align*}
    \widetilde{p}_{\ell+1} = \widetilde{r}_{\ell+1} + \dfrac{\widetilde{r}_{\ell+1}^\top \widetilde{r}_{\ell+1}}{\widetilde{r}_\ell^\top \widetilde{r}_\ell} \widetilde{p}_{\ell}
    &= \widebar Q_k^\top \widehat{r}_{\ell+1} + \dfrac{\widehat{r}_{\ell+1}^\top \widebar Q_k\widebar Q_k^\top \widehat{r}_{\ell+1}}{\widehat{r}_\ell^\top \widebar Q_k \widebar Q_k^\top \widehat{r}_\ell} \widebar Q_k^\top \widehat{p}_\ell\\
    &=\widebar Q_k^\top\left(\widehat{r}_{\ell+1} + \dfrac{\widehat{r}_{\ell+1}^\top \widehat{r}_{\ell+1}}{\widehat{r}_\ell^\top \widehat{r}_\ell}  \widehat{p}_\ell\right)
    =\widebar Q_k^\top \widehat{p}_{\ell+1}.
\end{align*}
Since $\widehat p_{\ell+1}$ is a linear combination of $\widehat p_\ell$ and $\widehat r_{\ell+1}$, it also belongs to $\mathcal R(\widehat A)$. Finally, $Q_kQ_k^\top\widehat p_{\ell}=0$, so the null-space component of $z_{\ell+1}$ remains equal to that of $z_0$. This proves~\eqref{CG-equivalence4} and completes the induction.
\end{proof}

In exact arithmetic, CG finds the exact solution $\widetilde{y}^\ast$ of~\eqref{eq:small_system} in at most $n-k$ iterations.
From~\eqref{CG-equivalence3} and~\eqref{CG-equivalence4}, the corresponding final iterate for~\eqref{eq:ProjSys} is
\begin{equation}\label{eq:explicite_sol}
    z^{\ast} = Q_k Q_k^\top z_0 + \widebar Q_k \widebar{\Delta}_k^{-1} \widebar Q_k^\top \widehat{b}.
\end{equation}
The first term in~\eqref{eq:explicite_sol} is the component of the initial vector $z_0$ in the null space of $\widehat{A}$. By~\eqref{CG-equivalence4}, in exact arithmetic, CG leaves it unchanged. The second term is the uniquely determined component in $\mathcal{R}(\widehat{A})$ that solves the projected system~\eqref{eq:ProjSys}.
For completeness, we now express the energy norm of the deflated CG error.
\begin{proposition}
\label{prop:deflatedCGTwoversio}
Let $1\leq \ell \leq n-k$. Assuming that the $\ell$-th deflated CG iterate $x_\ell^{\text{D}}$ exists, the corresponding error satisfies
\begin{equation*}
   \left\| x^\ast - x_\ell^{\text{D}}\right\|^2_{A} = \min_{p \in \mathbb{P}_{\ell}(0)}\lVert p (\widebar{\Delta}_k) \left( \widetilde{y}^{\ast} - \widetilde{y}_0\right) \rVert_{\widebar{\Delta}_k}^2 = \displaystyle \min_{p \in \mathbb{P}_{\ell}(0)}\lVert p (\widebar{\Delta}_k) \widebar Q_k^{\top} \left( z^{\ast} - z_0\right) \rVert_{\widebar{\Delta}_k}^2,
\end{equation*}
where $x^\ast$ represents the exact solution of $A x = b$ and $\widetilde{y}_0=\widebar Q_k^{\top}z_0$.
\end{proposition}

\begin{proof}
Applying~\eqref{eq:defcg_iterates_gecase} to both $x_\ell^{\mathrm D}$ and the exact solution $x^\ast$ (with $z_\ell$ replaced by $z^\ast$), and using
$P A P^\top = P A = \widebar Q_k \widebar{\Delta}_k \widebar Q_k^\top$ (see~\citep{frank2001construction}) gives
    \begin{align*}
    \left\| {x}^\ast -  {x}_\ell^{\text{D}}\right\|_{A}^2  = \left\| P^\top\left(z^{\ast} - z_\ell\right)\right\|_{A}^2 
    & = \left(z^{\ast} - z_\ell \right)^\top P A P^\top \left( z^{\ast} - z_\ell \right ) \\
    &= \left(z^{\ast} - z_\ell \right)^\top P A\left( z^{\ast} - z_\ell \right ) \\
    & = \left(z^{\ast} - z_\ell \right)^\top \widebar Q_k \widebar{\Delta}_k \widebar Q_k^\top \left( z^{\ast} - z_\ell \right ) \\
   & = \left\| \widebar Q_k^\top(z^{\ast} - z_\ell)\right\|_{\widebar{\Delta}_k}^2 = \left\|\widetilde{y}^{\ast} - \widetilde{y}_\ell\right\|^2_{\widebar{\Delta}_k}.
    \end{align*}
The CG minimization property for~\eqref{eq:small_system} now gives
\begin{equation*}
\left\|\widetilde{y}^{\ast} - \widetilde{y}_\ell\right\|^2_{\widebar{\Delta}_k} 
=\displaystyle \min_{p \in \mathbb{P}_{\ell}(0)}\lVert p (\widebar{\Delta}_k) \left( \widetilde{y}^{\ast} - \widetilde{y}_0\right) \rVert_{\widebar{\Delta}_k}^2
= \displaystyle \min_{p \in \mathbb{P}_{\ell}(0)}\lVert p (\widebar{\Delta}_k) \widebar Q_k^{\top} \left( z^{\ast} - z_0\right) \rVert_{\widebar{\Delta}_k}^2.
\end{equation*}
\end{proof}

%In the next section, we formulate deflated CG and PCG with the scaled spectral preconditioner~\citep{diouane2026spectral} for a perturbed system and develop the first-order perturbation analysis required to compare their energy-norm errors. 

\section{Deflated CG and PCG for a perturbed system}\label{sec:perturbed_system}
Let \( A(t) = A + tE \), where \( t \geq 0 \) and $E$ is symmetric with \( \|E\|_F = 1 \). Since $A$ is SPD, \( A(t) \) is SPD for sufficiently small values of \( t \).
In this section, we compare the scaled spectral preconditioner in~\citep{diouane2026spectral} with deflated CG when solving the perturbed linear system
\begin{equation}\label{eq:pert_sys}
    A(t) x = b.
\end{equation}

Let $1\leq k\leq n-1$. For deflated CG, we use the subspace $W=S_k=[s_1,\ldots,s_k]$, formed by the eigenvectors associated with the $k$ largest eigenvalues $\lambda_1,\lambda_2,\ldots,\lambda_k$ of $A$. Set $V(t)=A(t)S_k$ and $P(t)=I_n-V(t)\left(S_k^\top V(t)\right)^{-1}S_k^\top$.
The projected system is
\begin{equation}\label{eq:projected_system}
\widehat{A}(t) z = \widehat{b}(t),
\end{equation}
where $\widehat{A}(t):=P(t)A(t)$ and $\widehat{b}(t):=P(t)b$. The $\ell$-th iterate of deflated CG is given by
\begin{equation}
\label{eq:defcg_iterates}
  x_{\ell}^{[\mathrm{D}]}(t)
  = S_k \bigl(S_k^\top V(t)\bigr)^{-1} S_k^\top b + P(t)^\top z_\ell(t),
\end{equation} 
where $z_\ell(t)$ is the iterate of CG when solving~\eqref{eq:projected_system}.

Let $\theta>0$ and $\Lambda_k:=\diag(\lambda_1,\ldots,\lambda_k)$, the scaled spectral preconditioner is defined by $U^{\theta}=I_n+S_k\left(\sqrt{\theta}\,\Lambda_k^{-1/2}-I_k\right)S_k^\top$. The associated preconditioned system is
\begin{equation}\label{eq:precond_pert_sys}
    U^{\theta}A(t)U^{\theta} y = U^{\theta}b,
    \qquad x=U^{\theta}y.
\end{equation}
Similarly, the $\ell$th iterate of PCG is $x_\ell^{[P]}(t)=U^\theta y_\ell(t)$, where $y_\ell(t)$ is the iterate of CG when solving~\eqref{eq:precond_pert_sys} starting from $y_0$.
\subsection{Perturbation analysis}
In this section, we introduce the perturbation results to quantify how eigenvalues and eigenvectors change with $t$ and to control the gap terms in our subsequent comparison of the energy-norm errors of PCG and deflated CG.
We first recall the first-order perturbation theory for simple eigenvalues and eigenvectors of symmetric matrices. These classical results (see, e.g.,~\citep{stewart2001matrix}) describe how the eigenpairs $\{\lambda_{i}(t),s_i(t)\}$ of \( A(t) \) evolve with small $t$. For $i=1,\ldots,n$ and if $\lambda_i$ is simple, then
\begin{equation}
     \lambda_{i}(t) = \lambda_{i} + t s_i^\top E s_i + O(t^2), \quad
     s_i(t) = s_{i} + t \underline{S}_{i} (\lambda_i I_{n-1} - \underline{\Lambda}_i)^{-1} \underline{S}_{i}^\top E s_i + O(t^2), \label{eq:FirstOrder}
\end{equation}
where \(\underline{\Lambda}_i=\diag(\lambda_j)_{j\neq i}\), and \(\underline{S}_i\in\mathbb{R}^{n\times(n-1)}\) is obtained by removing the $i$th column from the eigenvector matrix of $A$.
Second-order remainder estimates for eigenpairs are known;
see, e.g.,~\citep[Ch.~II]{Kato1995}. For our analysis, we provide a short self-contained proof for symmetric matrices. The resulting bound involves only concrete constants and an explicit range of validity for $t$, both expressed in terms of the spectral gap, without relying on contour integration.

\begin{proposition}\label{prop:gap_bound}
Let $i\in\{1,\ldots,n\}$, and assume that $\lambda_i$ is simple. Let $g_i:=\min_{j\neq i}|\lambda_i-\lambda_j|>0$
be the spectral gap around $\lambda_i$. Then, for every $t \in [0, g_i/(4\|E\|_2)]$, $A(t)$ admits a simple eigenpair $\{\lambda_i(t),s_i(t)\}$ such that
{\small
\begin{equation*}%\label{eq:gap_remainder:2}
\bigl|\lambda_i(t)-\lambda_i-t\,s_i^\top E s_i\bigr|\leq\frac{4\|E\|_2^2}{g_i}\,t^2,\quad
\Bigl\|s_i(t)-s_i-t\underline S_i(\lambda_i I_{n-1}-\underline\Lambda_i)^{-1}\underline S_i^\top Es_i\Bigr\|_2\leq\frac{12\|E\|_2^2}{g_i^2}\,t^2.
\end{equation*}
}
\end{proposition}
\begin{proof}
We first establish an a priori bound. By Weyl's inequality for symmetric matrices (see, e.g.,~\citep{bhatia2013matrix,horn2012matrix}), $|\lambda_i(t)-\lambda_i|\leq t\,\|E\|_2$ for every $t$. Hence, for $0 \leq t\leq g_i/(2\|E\|_2)$, $|\lambda_i(t)-\lambda_i|\leq g_i/2$. Since, for every $j\neq i$, $|\lambda_i-\lambda_j|\geq g_i$, this gives
\begin{equation}\label{eq:gap_to_unperturbed_spectrum}
|\lambda_i(t)-\lambda_j|
\geq |\lambda_i-\lambda_j|-|\lambda_i(t)-\lambda_i|
\geq \frac{g_i}{2}.
\end{equation}
Consequently,~\eqref{eq:gap_to_unperturbed_spectrum} implies that, for $0\leq t\leq g_i/(2\|E\|_2)$,
{\small
\begin{equation}\label{eq:resolvent_bound_gap}
\bigl\|(\lambda_i(t)I_{n-1}-\underline\Lambda_i)^{-1}\bigr\|_2\leq\frac2{g_i}.
\end{equation}}
Let $0\leq t\leq g_i/(4\|E\|_2)$, we next show that $\lambda_i(t)$ remains simple.  For $j\neq i$,
{\small
\begin{equation*}
|\lambda_i(t)-\lambda_j(t)|
\geq |\lambda_i-\lambda_j|
   -|\lambda_i(t)-\lambda_i|
   -|\lambda_j(t)-\lambda_j|
\geq |\lambda_i-\lambda_j|-2t\|E\|_2
\geq g_i-2t\|E\|_2
\geq \frac{g_i}{2},
\end{equation*}
where we have used in the last inequality that $0\leq t\leq g_i/(4\|E\|_2)$. Thus $\lambda_i(t)$ is simple.} 
We now derive exact equations for the perturbed eigenpair. Let $s_i(t)$ be a unit eigenvector associated with $\lambda_i(t)$ and write $s_i(t)=a(t)s_i+\underline S_i d(t)$.
Its eigenvalue equation is
{\small
\begin{equation}\label{eq:perturbed_eigenpair_equation}
(A+tE)s_i(t)=\lambda_i(t)s_i(t).
\end{equation}}
Multiplying~\eqref{eq:perturbed_eigenpair_equation} by $\underline S_i^\top$ and using
$\underline S_i^\top As_i=0$,
$\underline S_i^\top A\underline S_i=\underline\Lambda_i$, and
$\underline S_i^\top s_i=0$, gives
\begin{equation}\label{eq:d_projected_equation}
\bigl(\lambda_i(t)I_{n-1}-\underline\Lambda_i\bigr)d(t)
=t\,\underline S_i^\top E s_i(t).
\end{equation}
The matrix in~\eqref{eq:d_projected_equation} is invertible by~\eqref{eq:resolvent_bound_gap}, and hence
{\small
\begin{equation}\label{eq:d_exact_formula}
d(t)=t\bigl(\lambda_i(t)I_{n-1}-\underline\Lambda_i\bigr)^{-1}
\underline S_i^\top E s_i(t).
\end{equation}}
Because $\underline S_i$ has orthonormal columns and $\|s_i(t)\|_2=1$, taking norms in~\eqref{eq:d_exact_formula} and applying~\eqref{eq:resolvent_bound_gap} yields,
\begin{equation}\label{eq:d_first_bound}
\|d(t)\|_2
\leq t\frac{2}{g_i}\|\underline S_i^\top E s_i(t)\|_2
\leq \frac{2t\|E\|_2}{g_i}.
\end{equation}
Since $ a(t)^2+\|d(t)\|_2^2=1$, combining~\eqref{eq:d_first_bound}  and  $\frac{2t\|E\|_2}{g_i}\leq\frac12$, gives $|a(t)|\geq\sqrt{3}/2>1/2$.

The eigenvalue and eigenvector gap can be bounded as follows. Multiplying~\eqref{eq:perturbed_eigenpair_equation} by $s_i^\top$ gives
{\small
\begin{equation}\label{eq:eigenvalue_remainder_identity}
a(t)\bigl(\lambda_i(t)-\lambda_i-t\,s_i^\top Es_i\bigr)
=t\,s_i^\top E\underline S_i d(t).
\end{equation}}
Because $|a(t)|>1/2$, combining~\eqref{eq:eigenvalue_remainder_identity} and~\eqref{eq:d_first_bound} gives,
{\small\begin{equation*}\label{eq:eigenvalue_remainder_bound}
\bigl|\lambda_i(t)-\lambda_i-t\,s_i^\top Es_i\bigr|
\leq \frac{1}{|a(t)|} t\|E\|_2\|d(t)\|_2
\leq 2t\|E\|_2\|d(t)\|_2
\leq\frac{4\|E\|_2^2}{g_i}\,t^2.
\end{equation*}}
On the other hand,
\begin{equation}\label{eq:triangular-inequality}
    \bigl\|s_i(t)-s_i
-t\underline S_iR_0\underline S_i^\top Es_i\bigr\|_2
\leq |a(t)-1|
+\bigl\|d(t)-tR_0\underline S_i^\top Es_i\bigr\|_2,
\end{equation}
where $R_0:=\bigl(\lambda_i I_{n-1}-\underline\Lambda_i\bigr)^{-1}$. We bound the two terms on the right-hand side of~\eqref{eq:triangular-inequality} separately.
Choose the sign of $s_i(t)$ so that $a(t)>0$. Since $a(t)^2+\|d(t)\|_2^2=1$,
\begin{equation}\label{eq:a_first_bound}
0\leq1-a(t)
=\frac{\|d(t)\|_2^2}{1+a(t)}
\leq\|d(t)\|_2^2
\leq\frac{4t^2\|E\|_2^2}{g_i^2}.
\end{equation}
Define $R(t):=\bigl(\lambda_i(t)I_{n-1}-\underline\Lambda_i\bigr)^{-1}$.
Using the expression of $s_i(t)$ in~\eqref{eq:d_exact_formula} gives
{\small
\begin{equation}\label{eq:d_exact_expanded}
d(t)=tR(t)\left(a(t)\underline S_i^\top Es_i
+\underline S_i^\top E\underline S_i d(t)\right).
\end{equation}}
Subtracting $tR_0\underline S_i^\top Es_i$ from~\eqref{eq:d_exact_expanded} yields
{\small
\begin{equation}\label{eq:d_remainder_identity}
d(t)-tR_0\underline S_i^\top Es_i
=t\bigl(R(t)-R_0\bigr)\underline S_i^\top Es_i+tR(t)\left((a(t)-1)\underline S_i^\top Es_i
+\underline S_i^\top E\underline S_i d(t)\right).
\end{equation}}
The resolvent identity gives 
%\begin{equation}\label{eq:local_resolvent_identity}
$R(t)-R_0
=\bigl(\lambda_i-\lambda_i(t)\bigr)R(t)R_0$.
%\end{equation}
Since $|\lambda_i - \lambda_i(t)|\leq t\|E\|_2$, $\|R(t)\|_2\leq2/g_i$ and $\|R_0\|_2=1/g_i$, we thus deduce,
{\small
\begin{equation}\label{eq:resolvent_bound}
\|R(t)-R_0\| \leq \frac{2t\|E\|_2}{g_i^2}.
\end{equation}
}
Using the triangular inequality in~\eqref{eq:d_remainder_identity}, and~\eqref{eq:resolvent_bound},~\eqref{eq:d_first_bound} and~\eqref{eq:a_first_bound}, we obtain
{\small
\begin{equation*}\label{eq:d_remainder_bound}
\begin{aligned}
\bigl\|d(t)-tR_0\underline{S}_i^\top E s_i\bigr\|_2
&\leq
t\|R(t)-R_0\|_2\|E\|_2
+t\|R(t)\|_2
\left(
|a(t)-1|\|E\|_2
+\|E\|_2\|d(t)\|_2
\right)\\
&\leq
\frac{2t^2\|E\|_2^2}{g_i^2}
+\frac{8t^3\|E\|_2^3}{g_i^3}
+\frac{4t^2\|E\|_2^2}{g_i^2}\leq
\frac{8t^2\|E\|_2^2}{g_i^2}.
\end{aligned}
\end{equation*}
}
where we have used $\|\underline S_i^\top Es_i\|\leq \|E\|_2$, $\|\underline S_i^\top E\underline S_i\|\leq \|E\|_2$, and the last inequality uses $t\|E\|_2/g_i\leq1/4$.
Combining this together with to~\eqref{eq:triangular-inequality},~\eqref{eq:a_first_bound} and~\eqref{eq:d_remainder_bound} completes the proof.
\end{proof}

To keep the analysis simple, in the remainder of the paper, we make the following assumption on the matrix $A$.
\begin{assumption}\label{ass:simple_spectrum}
The eigenvalues of the matrix $A$ are all simple.
\end{assumption}

\subsection{Deflated CG for the perturbed system}
In this section, we analyze deflated CG on the perturbed system~\eqref{eq:pert_sys}. %We first derive an explicit expression for its energy-norm error, together with first-order expansions of the projected matrix and its nonzero eigenpairs. We then expand the eigenpairs of the preconditioned matrix and relate its residual coefficients to those of deflated CG, which yields a bound on the PCG error in terms of the deflated CG error.

%Let $1 \leq k \leq n-1$,  and $\widebar{\Delta}_k(t)$ and $\widebar Q_k(t)$ be as in~\Cref{Sec:DefCG}. For convenience in the perturbation analysis, we index the nonzero eigenpairs of $\widehat A(t)$ by $i=k+1,\ldots,n$, consistently with the indexing of the corresponding eigenpairs of the reference matrix $A$. The next theorem gives an explicit expression for the energy norm of the error in deflated CG, based on~\Cref{prop:deflatedCGTwoversio} and~\Cref{lemma:CG_xproj_yproj}. 

Let $1\leq k\leq n-1$. For sufficiently small $t\geq0$, let $(\widehat{\lambda}_i(t),\widehat{s}_i(t))$, $i=k+1,\ldots,n$, denote the nonzero eigenpairs of $\widehat{A}(t)=P(t)A(t)$, where the eigenvectors are chosen orthonormal and the pairs are indexed by continuation from the eigenpairs $(\lambda_i,s_i)$ of $A$ at $t=0$. Collecting them, we define
\begin{equation*}
\widebar Q_k(t)=[\widehat{s}_{k+1}(t),\ldots,\widehat{s}_n(t)]
\quad\text{and}\quad
\widebar{\Delta}_k(t)=\diag(\widehat{\lambda}_{k+1}(t),\ldots,\widehat{\lambda}_n(t)),
\end{equation*}
so that $\widehat{A}(t)=\widebar Q_k(t)\,\widebar{\Delta}_k(t)\,\widebar Q_k(t)^\top$. The next theorem gives an explicit expression for the energy-norm error of deflated CG, based on~\Cref{prop:deflatedCGTwoversio} and~\Cref{lemma:CG_xproj_yproj}.

\begin{theorem}\label{thm:err_def_Dk}
Let $1\leq\ell<n-k$ and let $z_0$ be any starting vector. Assume that the deflated CG iterate $x_{\ell}^{[\text{D}]}(t)$ is well defined for sufficiently small $t\geq0$. Then,
\begin{equation*}
   \left\| x^\ast(t) - x_{\ell}^{[\text{D}]}(t)\right\|_{A(t)}^{2} = \sum_{i =k+1}^{n} \dfrac{\widetilde{\zeta}_i(t)^2}{\widehat{\lambda}_i(t)} p_{\ell}^{(t)}(\widehat{\lambda}_i(t))^2,
\end{equation*}
where $\widetilde{\zeta}_i(t):=\widehat{s}_i(t)^{\top}\left(\widehat{b}(t)-\widehat{A}(t) z_0\right)$, $x^\ast(t)$ is the exact solution of~\eqref{eq:pert_sys}, and
\begin{equation*}
\{p_{\ell}^{(t)}\}=\displaystyle\arg \min_{p \in \mathbb{P}_{\ell}(0)}\lVert p (\widebar{\Delta}_k(t) ) {\widebar Q_k(t)}^\top( z^{\ast}(t) - z_0) \rVert_{\widebar{\Delta}_k(t)}^2.
\end{equation*}
\end{theorem}
\begin{proof}
By~\Cref{prop:deflatedCGTwoversio} and~\Cref{lemma:CG_xproj_yproj}, applied to~\eqref{eq:projected_system}, we have
{\small
\begin{equation}\label{eq:p_l^t}
\begin{aligned}
\left\|x^\ast(t)-x_{\ell}^{[\text{D}]}(t)\right\|_{A(t)}^2
&=\min_{p\in\mathbb{P}_{\ell}(0)}
\left\|p(\widebar{\Delta}_k(t))\widebar Q_k(t)^\top\left(z^\ast(t)-z_0\right)\right\|_{\widebar{\Delta}_k(t)}^2\\
&=
\left\|p_{\ell}^{(t)}(\widebar{\Delta}_k(t))\widebar Q_k(t)^\top\left(z^\ast(t)-z_0\right)\right\|_{\widebar{\Delta}_k(t)}^2,    
\end{aligned}
\end{equation}}
where $z^\ast(t)$ denotes a solution of~\eqref{eq:projected_system}. The component of $z^\ast(t)$ in $\mathcal{R}(\widebar Q_k(t))$ satisfies
\(
\widebar Q_k(t)^\top z^\ast(t)
=\widebar{\Delta}_k(t)^{-1}\widebar Q_k(t)^\top\widehat{b}(t).
\)
Therefore,
\begin{equation*}
\widebar Q_k(t)^\top\left(z^\ast(t)-z_0\right)
=\widebar{\Delta}_k(t)^{-1}\widebar Q_k(t)^\top\widehat{b}(t)-\widebar Q_k(t)^\top z_0
=\widebar{\Delta}_k(t)^{-1}\widebar Q_k(t)^\top(\widehat{b}(t)-\widehat{A}(t)z_0).
\end{equation*}
where we used $\widehat{A}(t) =\widebar Q_k(t) \widebar{\Delta}_k(t) \widebar Q_k(t)^\top$. Substituting this identity into~\eqref{eq:p_l^t},
{\small\begin{align*}
\left\|x^\ast(t)-x_{\ell}^{[\text{D}]}(t)\right\|_{A(t)}^2
&=\left\|p_{\ell}^{(t)}(\widebar{\Delta}_k(t))\widebar{\Delta}_k(t)^{-1}\widebar Q_k(t)^\top(\widehat{b}(t)-\widehat{A}(t)z_0)\right\|_{\widebar{\Delta}_k(t)}^2\\
&=\sum_{i=k+1}^{n}\frac{\widetilde{\zeta}_i(t)^2}{\widehat{\lambda}_i(t)}
p_{\ell}^{(t)}(\widehat{\lambda}_i(t))^2.
\end{align*}}
\end{proof}    

We now present the asymptotic expansion of \( \widehat{A}(t) = P(t) A(t) \), along with its associated nonzero eigenvalues and corresponding eigenvectors, as well as \( \widehat{b}(t) = P(t) b \).

\begin{lemma}\label[lemma]{lm:Btilde}
Under~\Cref{ass:simple_spectrum}, for sufficiently small $t\geq0$, 
\begin{align*}
&\widehat{A}(t) = \widebar{S}_k \widebar\Lambda_{k} \widebar{S}_k^{\top} + t \widebar{S}_k \widebar{S}_k^{\top} E \widebar{S}_k \widebar{S}_k^{\top} + O(t^2), & \\
&\widehat{\lambda}_i(t) = \lambda_i(t) + O(t^2), & \text{for} \quad i = k+1, \ldots, n, \\
&\widehat{s}_i(t) = s_i + t \underline{S}_i D_i \underline{S}_i^\top E s_i + O(t^2), & \text{for} \quad i = k+1,\ldots, n, \\
&
\widehat{b}(t) = \widebar{S}_k \widebar{S}_k^\top b - t \widebar{S}_k \widebar{S}_k^{\top} E S_k \Lambda^{-1}_k S_k^{\top} b + O(t^2), 
\end{align*}
where $\widebar\Lambda_{k} :=\diag(\lambda_{k+1},\ldots,\lambda_n)$ and $D_i$ is a diagonal matrix such that
{\small\[
[D_i]_{jj} = \begin{cases}
0, \qquad\qquad\quad\; \text{for} \quad j = 1, \ldots, k, \\
1/(\lambda_i - \lambda_j), \quad \text{for} \quad j = k+1, \ldots, n : j\neq i.
\end{cases}
\]}

\end{lemma}
\begin{proof}
We first expand the projector $P(t)$. Since
\(S_k^\top V(t)=S_k^\top (A+tE)S_k=
\Lambda_k+tS_k^\top E S_k\),
a first-order expansion of the inverse gives
\begin{equation}
\label{eq:1st_order_inverse}
(S_k^\top V(t))^{-1}
=\Lambda_k^{-1}-t\Lambda_k^{-1}S_k^\top E S_k\Lambda_k^{-1}+O(t^2).
\end{equation}
Using this expansion in the definition of $P(t)$ yields
\begin{align*}
P(t)
&=I_n-V(t)(S_k^\top V(t))^{-1}S_k^\top 
\nonumber\\
&=I_n-(S_k\Lambda_k+tES_k)
  (\Lambda_k^{-1}-t\Lambda_k^{-1}S_k^\top E S_k\Lambda_k^{-1})S_k^\top+O(t^2)
\nonumber\\
&=I_n-S_kS_k^\top
  +tS_kS_k^\top E S_k\Lambda_k^{-1}S_k^\top
  -tES_k\Lambda_k^{-1}S_k^\top+O(t^2)
\nonumber\\
&=\widebar S_k\widebar S_k^\top
  -t(I_n-S_kS_k^\top)ES_k\Lambda_k^{-1}S_k^\top+O(t^2).
\nonumber
\end{align*}
Hence, $P(t) =\widebar S_k\widebar S_k^\top -t\widebar S_k\widebar S_k^\top ES_k\Lambda_k^{-1}S_k^\top+O(t^2)$.
We now use the expression of $P(t)$ to expand $\widehat{A}(t)$ and $\widehat{b}(t)$,
\begin{align*}
\widehat{A}(t)
=P(t)A(t)&=
\widebar S_k\widebar S_k^\top(A+tE)
-t\widebar S_k\widebar S_k^\top ES_k\Lambda_k^{-1}S_k^\top(A+tE)+O(t^2)\\
&=
\widebar S_k\widebar\Lambda_{k}\widebar S_k^\top
+t\widebar S_k\widebar S_k^\top E(I_n-S_kS_k^\top)+O(t^2)\\
&=
\widebar S_k\widebar\Lambda_{k}\widebar S_k^\top
+t\widebar S_k\widebar S_k^\top E\widebar S_k\widebar S_k^\top+O(t^2).
\end{align*}
Similarly,
$
\widehat{b}(t)
=P(t)b=
%(\widebar S_k\widebar S_k^\top -t\widebar S_k\widebar S_k^\top ES_k\Lambda_k^{-1}S_k^\top+O(t^2))b\\&=
\widebar S_k\widebar S_k^\top b
-t\widebar S_k\widebar S_k^\top ES_k\Lambda_k^{-1}S_k^\top b+O(t^2).
$

%This proves the expansions for $\widehat{A}(t)$ and $\widehat{b}(t)$.
Regarding the nonzero eigenvalues of $\widehat{A}(t)$, the matrix $\widehat A(t)=P(t)A(t)$, is analytic for all sufficiently small $t$. Standard perturbation theory  therefore applies to each simple nonzero eigenvalue of $\widehat A(0)$. In particular, the $O(t^2)$ term in the matrix expansion above does not affect the first-order eigenvalue and eigenvector coefficients. The unperturbed matrix
\(
\widehat{A}(0)=\widebar S_k\widebar\Lambda_{k}\widebar S_k^\top
\)
has, by~\Cref{ass:simple_spectrum}, $n-k$ simple eigenpairs $(\lambda_i,s_i)$ for $i=k+1,\ldots,n$. Applying the first-order perturbation~\eqref{eq:FirstOrder} to $\widehat{A}(t)$
gives, for $i=k+1,\ldots,n$,
\begin{align*}
\widehat{\lambda}_i(t)
&=\lambda_i+t s_i^\top \widebar S_k\widebar S_k^\top E\widebar S_k\widebar S_k^\top s_i+O(t^2)=\lambda_i+t s_i^\top E s_i+O(t^2)
=\lambda_i(t)+O(t^2),
\end{align*}
where we used $\widebar S_k\widebar S_k^\top s_i=s_i$.
Let \(
W_i=\diag(\underbrace{0,\ldots,0}_{k \text{ times}},\lambda_{k+1},\ldots,\lambda_{i-1},\lambda_{i+1},\ldots,\lambda_n)
\), then the same first-order perturbation formula for the eigenvectors gives
{\small
\begin{align*}
\widehat{s}_i(t)
&=s_i+t\underline S_i
(\lambda_i I_{n-1}- W_i)^{-1}
\underline S_i^\top\widebar S_k\widebar S_k^\top E\widebar S_k\widebar S_k^\top s_i+O(t^2)=s_i+t\underline S_iD_i\underline S_i^\top E s_i+O(t^2).
\end{align*}
}
%where\(vW_i=\diag(\underbrace{0,\ldots,0}_{k \text{ times}},\lambda_{k+1},\ldots,\lambda_{i-1},\lambda_{i+1},\ldots,\lambda_n)\). 
\end{proof}  

\subsection{Perturbation analysis of the preconditioned matrix}

The perturbation expansions for the eigenvectors of $U^\theta A(t)U^\theta$ involve the factors $1/(\lambda_i-\theta)$ for $i=k+1,\ldots,n$, so obtaining estimates that are uniform in the scaling parameter requires $\theta$ to stay a fixed positive distance away from the nondeflated eigenvalues. We formalize this through the following set.

\begin{definition}\label{ass:scalling}
Let $\gamma_0$ be the smallest spectral gap $
\gamma_0 := \min_{1\leq i\leq n-1}(\lambda_i-\lambda_{i+1})$ which is positive by~\Cref{ass:simple_spectrum}. For a fixed $k\in\{1,\ldots,n-1\}$ and $\delta>0$, define
\begin{equation}\label{eq:Theta_delta}
\Theta_\delta
:=\left\{\theta\in[\lambda_n,\lambda_k] \mid
\min_{i=k+1,\ldots,n}|\theta-\lambda_i|\geq\delta\right\}.
\end{equation}
\end{definition}

Note that, since the eigenvalues of $A$ are ordered decreasingly,
\[
\min_{i=k+1,\ldots,n}|\lambda_k-\lambda_i|
=\lambda_k-\lambda_{k+1}\geq\gamma_0.
\]
Thus, $\lambda_k\in\Theta_{\gamma_0}$. Decreasing $\delta$ relaxes the separation condition, so $\Theta_{\gamma_0}\subseteq\Theta_\delta$ whenever $0<\delta\leq\gamma_0$. Hence, $\lambda_k\in\Theta_\delta$ and $\Theta_\delta$ is nonempty.
%\end{proof}
We henceforth fix $0<\delta\leq\gamma_0$ and choose $\theta\in\Theta_\delta$, ensuring a distance of at least $\delta$ from $\lambda_{k+1},\ldots,\lambda_n$.

Throughout the remainder of the paper, an estimate $R(t,\theta)=O(t^2)$ stated uniformly for $\theta\in\Theta_\delta$ means that there exist constants $C>0$ and $\hat{t}>0$, both independent of $\theta$, such that
\[
\|R(t,\theta)\|\leq C \, t^2,
\quad
\text{for every }\theta\in\Theta_\delta\text{ and } 0\leq t\leq \hat{t}.
\]

We now turn to the asymptotic analysis of the eigenvalues and eigenvectors of the preconditioned matrix $U^{\theta}A(t)U^{\theta}$. For conciseness, set
\[
A^\theta(t):=U^\theta A(t)U^\theta,
\qquad
\lambda_i^\theta(t):=\lambda_i(A^\theta(t)),
\qquad
s_i^\theta(t):=s_i(A^\theta(t)).
\]
In this analysis, eigenvalue indices denote local perturbation branches rather than decreasing numerical order: indices $1,\ldots,k$ label the branch issued from the eigenvalue $\theta$, while each index $i=k+1,\ldots,n$ labels the branch issued from $\lambda_i$. %We apply~\Cref{prop:gap_bound} to the latter simple eigenvalues with $A$ replaced by $A^{\theta}(0) := U^\theta AU^\theta$; the eigenvalue $\theta$ of multiplicity $k$ is handled by Weyl's inequality.

\begin{lemma}\label{lem:lambdatilda}
Let~\Cref{ass:simple_spectrum} hold and fix $\delta \in (0,\gamma_0]$. Let $1\leq k\leq n-1$. For sufficiently small $t\geq0$, the eigenpairs of the preconditioned matrix indexed by $i=k+1,\ldots,n$ admit the following expansions, uniformly in $\theta\in\Theta_\delta$,
\begin{equation*}
\lambda_{i}^\theta(t) = \lambda_{i}(t) + O(t^2), \qquad
s_i^\theta(t) = s_i + t \underline{S}_i T_i^\theta \underline{S}_i^\top E s_i + O(t^2),
\end{equation*}
where $T_i^\theta$ is an $(n-1)\times(n-1)$ diagonal matrix such that
{\small
 $$[T_i^\theta]_{jj} = \begin{cases}
\sqrt{\theta}/(\sqrt{\lambda_j}(\lambda_i - \theta)), \quad \text{for}\quad j = 1, \ldots, k \\
1/(\lambda_i - \lambda_j), \quad \text{for}\quad j = k+1, \ldots, n : j\neq i.
\end{cases}$$}
\end{lemma}
\begin{proof}
Since
\(
A^\theta(t)=U^{\theta}AU^{\theta}+tU^{\theta}EU^{\theta}
\), the first-order perturbation~\eqref{eq:FirstOrder} gives
\begin{align*}
\lambda_i^\theta(t)
&=\lambda_i^\theta(0)
+t s_i^\top U^{\theta}EU^{\theta}s_i+O(t^2)
=\lambda_i+t s_i^\top E s_i+O(t^2)
= \lambda_i(t) + O(t^2),
\end{align*}
where we used $U^{\theta}s_i=s_i$ and $\lambda_i^\theta(0)=\lambda_i$ for $i=k+1,\ldots,n$. Applying the eigenvector expansion gives,
\begin{equation}\label{eq:s(UAU)}
s_i^\theta(t)
=s_i+t\underline S_i
\left(\lambda_i I_{n-1}-\Omega_i\right)^{-1}
\underline S_i^\top U^{\theta}EU^{\theta}s_i+O(t^2),
\end{equation}
where
\(
\Omega_i
=\diag(\underbrace{\theta,\ldots,\theta}_{k \text{ times}},
\lambda_{k+1},\ldots,\lambda_{i-1},\lambda_{i+1},\ldots,\lambda_n).
\)
Using $U^{\theta}s_i=s_i$ and 
\(
U^{\theta}\underline S_i
=\underline S_i\diag(
\frac{\sqrt{\theta}}{\sqrt{\lambda_1}},\ldots,
\frac{\sqrt{\theta}}{\sqrt{\lambda_k}},
\underbrace{1,\ldots,1}_{n-k-1 \text{ times}}),
\)
~\eqref{eq:s(UAU)} becomes $s_i^\theta(t)
=s_i+t\underline S_iT_i^\theta\underline S_i^\top E s_i+O(t^2)$.

For every $i>k$ and every $\theta\in\Theta_\delta$, the distance from $\lambda_i$ to the remaining spectrum of $U^\theta AU^\theta$ is at least $\delta$. Moreover, $\|U^\theta E U^\theta\|_2
\leq \|U^\theta\|_2^2\|E\|_2
\leq 1$.
Hence~\Cref{prop:gap_bound}, applied with the common gap $\delta$, provides a common perturbation radius and $\theta$-independent constants in the eigenvalue and eigenvector remainders.
\end{proof}

\subsection{Comparison between PCG and deflated CG}
This section presents the main comparison between the energy-norm errors of PCG and deflated CG. To connect the two error representations, we first compare their residual coefficients and then establish spectral bounds used in the main comparison theorem.

Throughout this section, we use the following assumptions.

\begin{assumption}\label{ass:z_0}
The initial guess $z_0$ when solving~\eqref{eq:projected_system} is chosen such that $S_k^\top r_0=0$, with $r_0 = b - A z_0$. A possible choice is $z_0=S_k\Lambda_k^{-1}S_k^\top b$.
\end{assumption}

\begin{assumption}\label{ass:y_0}
The initial guess for the preconditioned system~\eqref{eq:precond_pert_sys} is chosen such that $z_0 =U^\theta y_0$.
\end{assumption}

\begin{lemma}\label{lm:eta_beta}
Let~\Cref{ass:simple_spectrum} hold and fix $\delta \in (0,\gamma_0]$. Let $1\leq k \leq n-1$. For $i\in\{k+1,\ldots,n\}$, define
\begin{align*}
\widetilde{\zeta}_i(t) :=\widehat s_i(t)^\top
\bigl(\widehat b(t)-\widehat A(t)z_0\bigr)\quad \text{and} \quad \widetilde{\beta}_i^\theta(t) :=s_i^\theta(t)^\top
\bigl(U^\theta b-A^\theta(t) y_0\bigr).
\end{align*}
Then, under~\Cref{ass:y_0,ass:z_0}, for a $t\geq0$ sufficiently small, 
$\widetilde{\zeta}_i(t)^2=\widetilde{\beta}_i^\theta(t)^2+O(t^2)
$ uniformly in $\theta\in\Theta_\delta$. 
\end{lemma}
\begin{proof}
Using~\Cref{lm:Btilde},
{\small
\begin{flalign*}
& \widetilde{\zeta}_i(t) = \widehat{s}_i(t)^{\top}(\widehat{b}(t)-\widehat{A}(t) z_0) &&\\
& ~= (s_i + t \underline{S}_{i} D_i  \underline{S}_{i}^\top E s_i)^\top
  (\widebar{S}_k\widebar{S}_k^{\top} r_0
  -t\widebar{S}_k\widebar{S}_k^{\top}ES_k\Lambda^{-1}_kS_k^{\top}b
  -t\widebar{S}_k\widebar{S}_k^{\top}E\widebar{S}_k\widebar{S}_k^{\top}z_0)
  +O(t^2) &&\\
& ~= s_i^\top r_0 + t\left( s_i^\top E \underline{S}_{i} D_i  \underline{S}_{i}^\top r_0 - s_i^\top E S_k \Lambda^{-1}_k S_k^{\top} b - s_i^\top E \widebar{S}_k \widebar{S}_k^{\top} z_0\right) + O(t^2) &&\\
& ~= s_i^\top r_0 + t\left( s_i^\top E \underline{S}_{i} D_i  \underline{S}_{i}^\top r_0 - s_i^\top E S_k S_k^\top x^\ast - s_i^\top E \widebar{S}_k \widebar{S}_k^{\top} z_0\right) + O(t^2). &&
\end{flalign*}
}
Set $\eta_i:=s_i^\top r_0$. Then, squaring the expression  of $\widetilde{\zeta}_i(t)$ gives
{\small
\begin{equation}\label{eq:eta^2}
    \widetilde{\zeta}_i(t)^2 = \eta_i^2 + 2 t \eta_i \left( s_i^\top E \underline{S}_{i} D_i  \underline{S}_{i}^\top r_0 - s_i^\top E S_k S_k^\top x^\ast - s_i^\top E \widebar{S}_k \widebar{S}_k^{\top} z_0\right) + O(t^2).
\end{equation}}
Let $\widehat{T}_i^{\theta}$ be an $(n-1)\times (n-1)$ diagonal matrix such that
{\small
\begin{equation*}
    [\widehat{T}_i^{\theta}]_{jj} = \begin{cases}
\theta/(\lambda_j(\lambda_i - \theta)), \quad \text{for}\quad j = 1, \ldots, k, \\
1/(\lambda_i - \lambda_j), \qquad\; \text{for}\quad j = k+1, \ldots, n : j\neq i.
\end{cases}
\end{equation*}}
Then, 
by~\Cref{ass:y_0}, the preconditioned initial residual satisfies
\[
U^\theta b-A^\theta(t) y_0
=U^\theta\bigl(b-A(t)z_0\bigr)
=U^\theta r_0-tU^\theta Ez_0.
\]
Using this identity, the expression for $\widehat{T}_i^{\theta}$, and~\Cref{lem:lambdatilda}, we obtain
{\small
\begin{align*}
         \widetilde{\beta}_i^\theta(t) &= s_i^\theta(t)^{\top}
         \bigl(U^\theta b-A^\theta(t) y_0\bigr)\\
         &= \left(s_i + t \underline{S}_{i} T_i^\theta  \underline{S}_{i}^\top E s_i + O(t^2)\right)^\top \left(U^\theta r_0 - t U^\theta 
 E z_0\right)\\
         &= s_i^\top r_0 + t\left( s_i^\top E \underline{S}_{i} \widehat{T}_i^{\theta}  \underline{S}_{i}^\top r_0 - s_i^\top E z_0\right) + O(t^2)\\
         &= s_i^\top r_0 + t\left( s_i^\top E \underline{S}_{i} \widehat{T}_i^{\theta}  \underline{S}_{i}^\top r_0 - s_i^\top E S_k S_k^\top z_0 - s_i^\top E \widebar{S}_k \widebar{S}_k^\top z_0\right) + O(t^2).
\end{align*}}
Squaring this expression gives
\begin{equation}\label{eq:beta^2}
    \widetilde{\beta}_i^\theta(t)^2 = \eta_i^2 + 2 t \eta_i \left(s_i^\top E \underline{S}_{i} \widehat{T}_i^{\theta}  \underline{S}_{i}^\top r_0 - s_i^\top E S_k S_k^\top z_0 - s_i^\top E \widebar{S}_k \widebar{S}_k^\top z_0\right) + O(t^2).
\end{equation}
From~\eqref{eq:eta^2} and~\eqref{eq:beta^2},
{\small
\begin{align*}
    \widetilde{\beta}_i^\theta(t)^2 - \widetilde{\zeta}_i(t)^2 &= 2 t\eta_i \left(s_i^\top E \underline{S}_i (\widehat{T}_i^{\theta} - D_i)\underline{S}_i^\top r_0 + s_i^\top ES_k S_k^\top (x^\ast - z_0)\right)+O(t^2)\\
    &=2 t\eta_i \left(s_i^\top E \underline{S}_i (\widehat{T}_i^{\theta} - D_i)\underline{S}_i^\top r_0 + s_i^\top ES_k S_k^\top A^{-1} r_0\right)+O(t^2)\\
    &=2 t\eta_i \left(s_i^\top E S_k \widehat{D}_i^{\theta} S_k^\top r_0 + s_i^\top ES_k \Lambda_{k}^{-1}  S_k^\top r_0\right)+O(t^2)\\
    &=2 t\eta_i \left(s_i^\top E S_k \left(\widehat{D}_i^{\theta} + \Lambda_{k}^{-1}\right)S_k^\top r_0 \right)+O(t^2),
\end{align*}}
where we used
\(
\underline{S}_i(\widehat{T}_i^{\theta}-D_i)\underline{S}_i^\top
=S_k\widehat{D}_i^{\theta}S_k^\top,
\) with \(
\widehat{D}_i^{\theta}
:=\diag\left(\frac{\theta}{\lambda_j(\lambda_i-\theta)}\right)_{j=1,\ldots,k}.
\)
By~\Cref{ass:z_0}, $S_k^\top r_0=0$, so the entire coefficient of $t$ vanishes. It remains to verify that the resulting $O(t^2)$ estimate is uniform in $\theta$.  In fact, first,~\Cref{ass:scalling} gives $\lambda_n\leq\theta\leq\lambda_k$ and $|\lambda_i-\theta|\geq\delta$ for $i=k+1,\ldots,n$.
Consequently, for any $j\leq k$, $\left|
\frac{\theta}{\lambda_j(\lambda_i-\theta)}
\right|
\leq \frac{1}{\delta}$, and $
\left|
\frac{\sqrt{\theta}}{\sqrt{\lambda_j}(\lambda_i-\theta)}
\right|
\leq \frac{1}{\delta}$,
because $\theta\leq\lambda_k\leq\lambda_j$. Thus, the matrices $\widehat T_i^\theta$ and $T_i^\theta$ are uniformly bounded for any $\theta\in\Theta_\delta$. Second, $\|U^\theta\|_2\leq1$ uniformly. Third, the second-order eigenvector remainder in~\Cref{lem:lambdatilda} is uniform over $\Theta_\delta$. All factors produced when the two expansions are multiplied and squared are therefore bounded independently of $\theta$. Moreover, by~\Cref{prop:gap_bound} and the argument in~\Cref{lem:lambdatilda}, the threshold underlying the eigenvector remainder is independent of $\theta$; hence the threshold for the resulting $O(t^2)$ estimate is independent of $\theta$.
\end{proof}

The following lemma is useful for the proof of the main result.
\begin{lemma}\label{lm:spectral_intervals}
Let~\Cref{ass:simple_spectrum} hold and fix $\delta \in (0,\gamma_0]$. Let $1\leq k\leq n-1$ and $t_0>0$, such that $t_0\|E\|_2<\lambda_n$.
Then for every $0\leq t\leq t_0$:
\begin{enumerate}
\item[(i)] $[\lambda_n-t_0\|E\|_2,\lambda_{k+1}+t_0\|E\|_2]$ contains the nonzero spectrum of $\widehat A(t)$;
\item[(ii)] $[\lambda_n-t_0\|E\|_2,\lambda_{k}+t_0\|E\|_2]$ contains the spectrum of $A^\theta(t)$, for every $\theta\in\Theta_\delta$.
\end{enumerate}
\end{lemma}

\begin{proof}
First, let us present the projected matrix as a Schur complement. Since $[S_k,\widebar S_k]$ is an orthonormal basis of $\mathbb R^n$, write $A(t)$ in this basis as $M_{11} := S_k^\top A(t) S_k$, $M_{12} := S_k^\top A(t)\widebar S_k$, and $M_{22} := \widebar S_k^\top A(t)\widebar S_k$.
Using $A(t)S_k=S_kM_{11}+\widebar S_kM_{12}^\top$ and $S_k^\top A(t)=M_{11}S_k^\top+M_{12}\widebar S_k^\top$, the matrix $\widehat A(t)=A(t)-A(t)S_kM_{11}^{-1}S_k^\top A(t)$ expands to
{\small
\[
\widehat A(t)
=A(t)-\bigl(S_kM_{11}+\widebar S_kM_{12}^\top\bigr)M_{11}^{-1}\bigl(M_{11}S_k^\top+M_{12}\widebar S_k^\top\bigr)
=\widebar S_k\bigl(M_{22}-M_{12}^\top M_{11}^{-1}M_{12}\bigr)\widebar S_k^\top.
\]
}
Since $\widebar S_k$ has orthonormal columns, the nonzero spectrum of $\widehat A(t)$ therefore coincides with the spectrum of $M_{22}-M_{12}^\top M_{11}^{-1}M_{12}$, and it suffices to prove
\begin{equation}\label{eq:shur-eq}
    \widebar\Lambda_k-t_0\|E\|_2I_{n-k}\preceq M_{22}-M_{12}^\top M_{11}^{-1}M_{12}\preceq\widebar\Lambda_k+t_0\|E\|_2I_{n-k}.
\end{equation}

We use the following characterization of the Schur complement: For any $v\in\mathbb R^{n-k}$,
\[
v^\top (M_{22}-M_{12}^\top M_{11}^{-1}M_{12})v=\min_{z\in\mathbb R^k}(S_kz+\widebar S_kv)^\top A(t)(S_kz+\widebar S_kv).
\]
Indeed, expanding the quadratic form gives
\begin{equation}\label{eq:quadra-shur}
(S_kz+\widebar S_kv)^\top A(t)(S_kz+\widebar S_kv)
=z^\top M_{11}z+2z^\top M_{12}v+v^\top M_{22}v,    
\end{equation}
and, since $M_{11}$ is SPD, the minimum of~\eqref{eq:quadra-shur} with respect to $z$ is attained at $z=-M_{11}^{-1}M_{12}v$; substituting back gives $v^\top(M_{22}-M_{12}^\top M_{11}^{-1}M_{12})v$, as claimed.

We now prove the two bounds in~\eqref{eq:shur-eq}. Since $0\leq t\leq t_0$, we have $A-t_0\|E\|_2 I_n\preceq A(t)\preceq A+t_0\|E\|_2I_n$. For the lower bound, this and the characterization above give, for any $v\in\mathbb{R}^{n-k}$,
{\small
\begin{align*}
v^\top (M_{22}-M_{12}^\top M_{11}^{-1}M_{12})v
&\geq \min_{z\in\mathbb R^k}
 (S_kz+\widebar S_kv)^\top(A-t_0\|E\|_2I_n)(S_kz+\widebar S_kv)\\
&=\min_{z\in\mathbb R^k}
\left\{z^\top(\Lambda_k-t_0\|E\|_2I_k)z
+v^\top(\widebar\Lambda_k-t_0\|E\|_2I_{n-k})v\right\}\\
&=v^\top(\widebar\Lambda_k-t_0\|E\|_2I_{n-k})v,
\end{align*}}
where we used $S_k^\top A\widebar S_k=0$ and $S_k^\top\widebar S_k=0$, and where the minimum is attained at $z=0$ because $\Lambda_k-t_0\|E\|_2I_k\succ0$. For the upper bound, evaluating~\eqref{eq:quadra-shur} at $z=0$ and using $A(t)\preceq A+t_0\|E\|_2I_n$ gives, for any $v\in\mathbb{R}^{n-k}$,
\[
v^\top(M_{22}-M_{12}^\top M_{11}^{-1}M_{12})v
\leq v^\top\widebar S_k^\top A(t)\widebar S_kv
\leq v^\top(\widebar\Lambda_k+t_0\|E\|_2I_{n-k})v,
\]
where we used $\widebar S_k^\top A\widebar S_k=\widebar\Lambda_k$. Together, the two bounds establish~\eqref{eq:shur-eq}, and hence that the nonzero spectrum of $\widehat A(t)$ lies in $[\lambda_n-t_0\|E\|_2,\lambda_{k+1}+t_0\|E\|_2]$.

For $\theta\in\Theta_\delta$, the eigenvalues of $A^\theta(0)$ lie in $[\lambda_n,\lambda_k]$, and $\|U^\theta\|_2\leq1$. Hence
\[
\|A^\theta(t)-A^\theta(0)\|_2=t\,\|U^\theta EU^\theta\|_2\leq t\,\|E\|_2\leq t_0\|E\|_2.
\]
By Weyl's inequality we prove (ii).
\end{proof}

We now state our main result: a second-order comparison of the energy-norm errors of PCG with $U^\theta$ and deflated CG, with an explicit remainder.

% \begin{theorem}\label{theorem:Error_perturbed_case-f}
% Let $\theta > 0$, $1\leq k\leq n-1$ and $t > 0$. Assume that \Cref{ass:simple_spectrum,ass:z_0,ass:y_0} hold. For $\ell \geq 1$, let $x_{\ell}^{[P]}(t)=U^\theta y_\ell$ denote the $\ell$th iterate for~\eqref{eq:precond_pert_sys} and $x_\ell^{[\mathrm D]}(t)$ the one given by~\eqref{eq:defcg_iterates}. Assume that the iterate $\ell +1$ PCG and deflated CG have not reached yet the solution.
% Let $0<\delta\leq\gamma_0$.

% Then, there exist $t^*>0$, $\gamma_1>0$, and $\gamma_2>0$, all independent of $\ell$ and $k$, such that, for all $\theta\in\Theta_\delta$, and for any $0\leq t\leq t^*$,
% \begin{equation}\label{eq:upperbound_explicit_ell}
% \left\|x^\ast(t)-x_{\ell+1}^{[P]}(t)\right\|_{A(t)}^2
% \leq
% \frac{\alpha(t, \theta)^2}{\theta^2}
% \left\|x^\ast(t)-x_\ell^{[\mathrm D]}(t)\right\|_{A(t)}^2
% +\left(\gamma_1q_k^{2\ell}+\gamma_2\ell q_k^{2\ell-1}\right)t^2,
% \end{equation}
% where $\alpha(t, \theta)
% :=\max_{k+1\leq i\leq n}
% \left|\widehat\lambda_i(t)-\theta\right|$ and $q_k
% :=\frac{\lambda_k + t^*\|E\|_2}{\lambda_n - t^*\|E\|_2}-1.$
% \end{theorem}

\begin{theorem}\label{theorem:Error_perturbed_case}
Let~\Cref{ass:simple_spectrum} hold and fix $\delta \in (0,\gamma_0]$. Consider $k \in \{1,...,n-1\}$ and  $\theta\in\Theta_\delta$. Let $z_0$ and $y_0$ be as in~\Cref{ass:z_0,ass:y_0}. Then, there exist $t^\ast \in (0,\lambda_n/\|E\|_2)$,  $\gamma_1>0$ , and $\gamma_2>0$ such that, for every  $1\leq \ell\leq n-k$ and $t\in[0,t^\ast]$ such that $x_{\ell+1}^{[P]}(t)=U^\theta y_{\ell+1}$ and $x_\ell^{[\mathrm D]}(t)$ exist,
{\small
\begin{equation}\label{eq:upperbound_explicit_ell}
\left\|x^\ast(t)-x_{\ell+1}^{[P]}(t)\right\|_{A(t)}^2
\leq
\frac{\alpha(t, \theta)^2}{\theta^2}
\left\|x^\ast(t)-x_\ell^{[\mathrm D]}(t)\right\|_{A(t)}^2
+\left(\gamma_1q_k^{2\ell}+\gamma_2\ell q_k^{2\ell-1}\right)t^2,
\end{equation}}
where $\alpha(t, \theta)
:=\displaystyle\max_{k+1\leq i\leq n}
\left|\widehat\lambda_i(t)-\theta\right|$ and $q_k
:=\frac{\lambda_k + t^\ast\|E\|_2}{\lambda_n - t^\ast\|E\|_2}-1$.
\end{theorem}
 
\begin{proof}
The proof proceeds in three steps: (1) we split the PCG error into a clustered and a nonclustered contribution using an admissible test polynomial; (2) we bound the clustered contribution; and (3) we bound the nonclustered contribution using~\Cref{lm:eta_beta}.
Let $t_0>0$ be chosen as in~\Cref{lm:spectral_intervals}. A value $t^\ast>0$ that is uniform in $k$ and in $\theta\in\Theta_\delta$ is obtained by taking the minimum of $t_0$ and the finitely many positive  thresholds associated, for $1\leq k\leq n-1$, with~\Cref{lm:Btilde,lem:lambdatilda,lm:eta_beta}.

Fix $1\leq \ell \leq n-k$ and $t\in[0,t^\ast]$ for which both iterates exist. For $i=1,\ldots,n$, set $\widetilde{\beta}_i^\theta(t):=s_i^{\theta}(t)^\top\bigl(U^\theta b-A^{\theta}(t)y_0\bigr)$ and let $y^\ast(t)$ be the exact solution of $A^{\theta}(t)y=U^\theta b$. Then, $x^\ast(t)=U^\theta y^\ast(t)$ and $\left\|x^\ast(t)-U^\theta y\right\|_{A(t)}^2
=\left\|y^\ast(t)-y\right\|_{A^{\theta}(t)}^2$ for any $y$.
The standard CG minimization property therefore gives
\begin{equation}\label{eq:pcg_error_representation}
\left\|x^\ast(t)-x_{\ell+1}^{[P]}(t)\right\|_{A(t)}^2
=\min_{q\in\mathbb P_{\ell+1}(0)}
\sum_{i=1}^n
\frac{\widetilde{\beta}_i^\theta(t)^2}{\lambda_i^\theta(t)}
q\bigl(\lambda_i^\theta(t)\bigr)^2.
\end{equation}
%For $i=k+1,\ldots,n$, \Cref{lm:Btilde,lem:lambdatilda} imply
%\begin{equation}\label{eq:mu_nu_O2}
%\lambda_i^\theta(t)-\widehat{\lambda}_i(t)=O(t^2),
%\qquad
%\widehat{\lambda}_i(t)-\lambda_i(t)=O(t^2).
%\end{equation}
%The first estimate is uniform for $\theta\in\Theta_\delta$ by~\Cref{lem:lambdatilda}; the second does not involve $\theta$ and follows from~\Cref{lm:Btilde}.

Consider $q_{\ell+1}^{(t)}(\lambda)
=\left(1-\lambda/\theta\right)p_\ell^{(t)}(\lambda)
\in\mathbb P_{\ell+1}(0)$. Using this polynomial in~\eqref{eq:pcg_error_representation} and separating the sum over the eigenvalues of $A^{\theta}(t)$ into the clustered indices $i=1,\ldots,k$ (close to $\theta$) and the nonclustered indices $i=k+1,\ldots,n$ gives
{\small
\begin{equation}\label{eq:pcg_split}
\left\|x^\ast(t)-x_{\ell+1}^{[P]}(t)\right\|_{A(t)}^2
\leq T_{\mathrm c}^{\theta}(t)+T_{\mathrm n}^{\theta}(t),
\end{equation}}
where $T_{\mathrm c}^{\theta}(t)
=\sum_{i=1}^k
\widetilde{\beta}_i^\theta(t)^2 
f_{\ell,t}^{\theta}(\lambda_i^\theta(t))$ and $T_{\mathrm n}^{\theta}(t)
=\sum_{i=k+1}^n\widetilde{\beta}_i^\theta(t)^2
f_{\ell,t}^{\theta}(\lambda_i^\theta(t))$ with $f_{\ell,t}^{\theta}(\lambda) :=\frac{1}{\lambda}
\left(1-\frac{\lambda}{\theta}\right)^2 p_\ell^{(t)}(\lambda)^2$. For conciseness, define $\underline{\mu}:=\lambda_n-t^\ast\|E\|_2$ and $\widebar{\mu}_j:=\lambda_j+t^\ast\|E\|_2$ for $j\in\{k,k+1\}$.

We first bound $T_{\mathrm c}^{\theta}(t)$. Let $p_\ell^{(t)}(\lambda)
=\prod_{j=1}^\ell\left(1-\frac{\lambda}{\xi_j(t)}\right)$,
where $\xi_1(t),\ldots,\xi_\ell(t)$ are the Ritz values. Since these Ritz values lie in $[\underline{\mu},\widebar{\mu}_{k+1}]$, for every $\lambda\in[\underline{\mu},\widebar{\mu}_k]$,
{\small
\begin{equation}\label{eq:p_q_bound}
\left|p_\ell^{(t)}(\lambda)\right|\leq \max_{\substack{\lambda\in [\underline{\mu},\widebar{\mu}_k]\\ \xi\in[\underline{\mu},\widebar{\mu}_{k+1}]}}\left|1-\frac{\lambda}{\xi}\right|^\ell = \left(\dfrac{\lambda_k + t^\ast\|E\|_2}{\lambda_n - t^\ast\|E\|_2}-1 \right)^\ell=q_k^\ell,
\end{equation}}
The equality follows from the fact that the maximum is attained at $\lambda=\widebar{\mu}_k$ and $\xi=\underline{\mu}$ (using $\underline{\mu}>0$ and $\widebar{\mu}_k\geq\widebar{\mu}_{k+1}$). For $i\leq k$, Weyl's inequality gives $|\lambda_i^\theta(t)-\theta|\leq t$; since $\theta\geq\lambda_n$, this yields
$
\left|1-\frac{\lambda_i^\theta(t)}{\theta}\right|^2
\leq\frac{t^2}{\theta^2}
\leq\frac{t^2}{\lambda_n^2}$. Using~\eqref{eq:p_q_bound}, gives
{\small
\begin{align*}
T_{\mathrm c}^{\theta}(t)\leq \left(\sum_{i=1}^k
\frac{\widetilde{\beta}_i^\theta(t)^2}{\lambda_i^\theta(t)}\right)q_k^{2\ell}t^2/\lambda_n^2 
&\leq \frac{\left\|U^\theta\bigl(b-A(t)z_0\bigr)\right\|_2^2}
{\underline{\mu}\,\lambda_n^2}\,q_k^{2\ell}t^2\\
&\leq \frac{\displaystyle\max_{0\leq s\leq t^\ast}
\left\|b-A(s)z_0\right\|_2^2}
{\underline{\mu}\,\lambda_n^2}\,q_k^{2\ell}t^2
=:c_{1,k}q_k^{2\ell}t^2.
\end{align*}}
Here, we used $\lambda_i^\theta(t)\geq\underline{\mu}$ (\Cref{lm:spectral_intervals}) and $\|U^\theta\|_2\leq1$.
The constant $c_{1,k}$ depends on $k$ because the initial guess
$z_0$ depends on $k$. Since $k\in\{1,\ldots,n-1\}$, define $c_1:=\max_{1\leq j\leq n-1}c_{1,j}$.
Then 
\begin{equation}\label{eq:Tc_explicit_ell}
    T_{\mathrm c}^{\theta}(t)\leq c_1q_k^{2\ell}t^2.
\end{equation}

We next bound $T_{\mathrm n}^{\theta}(t)$. Differentiating $p_\ell^{(t)}$ and using $\xi_j(t)\geq\underline{\mu}$ gives, for $\lambda\in [\underline{\mu},\widebar{\mu}_{k}]$,
{\small
\begin{equation}\label{eq:p_prime_q_bound}
\left|\partial_\lambda p_\ell^{(t)}(\lambda)\right|= \left|-\sum_{j=1}^{\ell}\frac{1}{\xi_j(t)}
\prod_{\substack{h=1\\h\neq j}}^{\ell}
(1-\frac{\lambda}{\xi_h(t)})\right|
\leq \frac{\ell}{\underline{\mu}}q_k^{\ell-1}.
\end{equation}}
Write $f_{\ell,t}^{\theta}(\lambda)=g^\theta(\lambda)p_\ell^{(t)}(\lambda)^2$ with $g^\theta(\lambda)=\lambda^{-1}(1-\lambda/\theta)^2$. For $\lambda\in [\underline{\mu},\widebar{\mu}_k]$,
{\small
\[
|g^\theta(\lambda)|
=\frac1\lambda(1-\frac\lambda\theta)^2
\leq\frac1{\underline{\mu}}(1+\frac{\lambda_1+t^\ast\|E\|_2}{\lambda_n})^2=:c_2,
\;
|\partial_\lambda g^\theta(\lambda)|
=\left|-\frac1{\lambda^2}+\frac1{\theta^2}\right|
\leq\frac1{\underline{\mu}^2}+\frac1{\lambda_n^2}=:c_2'.
\]
}
Combining these estimates with~\eqref{eq:p_q_bound} and~\eqref{eq:p_prime_q_bound} yields $c_2,c_3,c_4>0$, independent of $\ell$ and of $\theta\in\Theta_\delta$, such that the following bounds hold for every $\lambda\in [\underline{\mu},\widebar{\mu}_k]$:
\begin{equation}\label{eq:f_q_bound}
|f_{\ell,t}^{\theta}(\lambda)|
\leq c_2 q_k^{2\ell},
\qquad |\partial_\lambda f_{\ell,t}^{\theta}(\lambda)|
\leq c_3 q_k^{2\ell} + c_4 \ell q_k^{2\ell-1}.
\end{equation}
For each $i=k+1,\ldots,n$ 
\begin{equation}\label{eq:nonclustered-term-decomposition}
\begin{aligned}
\widetilde{\beta}_i^\theta(t)^2f_{\ell,t}^{\theta}(\lambda_i^\theta(t))
={}&\widetilde\zeta_i(t)^2f_{\ell,t}^{\theta}(\widehat{\lambda}_i(t))+\bigl(\widetilde{\beta}_i^\theta(t)^2-\widetilde\zeta_i(t)^2\bigr)
  f_{\ell,t}^{\theta}(\lambda_i^\theta(t))\\
&+\widetilde\zeta_i(t)^2
  \bigl(f_{\ell,t}^{\theta}(\lambda_i^\theta(t))-f_{\ell,t}^{\theta}(\widehat{\lambda}_i(t))\bigr).
\end{aligned}
\end{equation}
\Cref{lm:eta_beta,lem:lambdatilda} provide constants
$c_{\beta, k, i},c_{\lambda,k, i}\geq0$ (independent of $\theta$) such that
\begin{equation*}\label{eq:clambdak}
\left|\widetilde\beta_i^\theta(t)^2-\widetilde\zeta_i(t)^2\right|
\leq c_{\beta, k, i}t^2,
\qquad
\left|\lambda_i^\theta(t)-\widehat\lambda_i(t)\right|
\leq c_{\lambda, k, i}t^2,
\qquad i=k+1,\ldots,n. 
\end{equation*}
Setting $c_\beta := \max\limits_{1\leq j\leq n-1 ,j+1\leq l\leq n} c_{\beta, j, l}$ and $c_\lambda := \max\limits_{1\leq j\leq n-1 ,j+1\leq l\leq n} c_{\lambda, j, l}$, yields
{\small
\begin{equation}\label{eq:clambda}
\left|\widetilde\beta_i^\theta(t)^2-\widetilde\zeta_i(t)^2\right|
\leq c_\beta t^2,
\qquad
\left|\lambda_i^\theta(t)-\widehat\lambda_i(t)\right|
\leq c_{\lambda}t^2,
\qquad i=k+1,\ldots,n. 
\end{equation}}
Using~\eqref{eq:f_q_bound}, the sum of the second terms
in~\eqref{eq:nonclustered-term-decomposition} satisfies
{\small
\begin{equation}\label{eq:sec_term}
    \sum_{i=k+1}^n
\left|\widetilde\beta_i^\theta(t)^2-\widetilde\zeta_i(t)^2\right|
\left|f_{\ell,t}^\theta(\lambda_i^\theta(t))\right|
\leq n c_{\beta}c_2q_k^{2\ell}t^2.
\end{equation}}
Set $c_{\zeta}:=\max\limits_{1\leq j\leq n-1} \max_{0\leq s\leq t^\ast} \sum_{i=j+1}^n\widetilde\zeta_i(s)^2$.~\eqref{eq:f_q_bound} and the mean value theorem yield
{\small
\begin{equation}\label{eq:third_term}
\begin{aligned}
\sum_{i=k+1}^n\widetilde\zeta_i(t)^2
\left|f_{\ell,t}^\theta(\lambda_i^\theta(t))
-f_{\ell,t}^\theta(\widehat\lambda_i(t))\right|
&\leq 
c_{\zeta} 
\max_{k+1\leq i\leq n }\left|f_{\ell,t}^\theta(\lambda_i^\theta(t))
-f_{\ell,t}^\theta(\widehat\lambda_i(t))\right|\\
&\leq
c_{\zeta}
\sup_{\lambda\in[\underline{\mu},\widebar{\mu}_k]}|\partial_\lambda f_{\ell,t}^{\theta}(\lambda)|\max_{k+1\leq i\leq n }\left|\lambda_i^\theta(t)-\widehat\lambda_i(t)\right|\\
&\leq
c_{\zeta}c_{\lambda}
\left(c_3q_k^{2\ell}+c_4\ell q_k^{2\ell-1}\right)t^2.
    \end{aligned}
\end{equation}}
Summing~\eqref{eq:nonclustered-term-decomposition} over $i=k+1,\ldots,n$, and combining~\eqref{eq:sec_term}-\eqref{eq:third_term}, gives 
{\small
\[
T_{\mathrm n}^{\theta}(t)
\leq \sum_{i=k+1}^n\widetilde\zeta_i(t)^2 f_{\ell,t}^{\theta}(\widehat{\lambda}_i(t))
+\left(c_5q_k^{2\ell}+c_6\ell q_k^{2\ell-1}\right)t^2,
\]}
where $c_5:=n c_\beta c_2+ c_\zeta c_\lambda c_3$ and $c_6:= c_\zeta c_\lambda c_4$. Using $\left|1-\widehat{\lambda}_i(t)/{\theta}\right|^2
\leq{\alpha(t, \theta)^2}/{\theta^2}$
and the deflated CG error representation in~\Cref{thm:err_def_Dk}, we obtain
{\small
\begin{equation}\label{eq:Tn_final}
T_{\mathrm n}^{\theta}(t)
\leq
\frac{\alpha(t, \theta)^2}{\theta^2}
\left\|x^\ast(t)-x_\ell^{[\mathrm D]}(t)\right\|_{A(t)}^2
+\left(c_5q_k^{2\ell}+c_6\ell q_k^{2\ell-1}\right)t^2.
\end{equation}}
The result follows from~\eqref{eq:pcg_split}, \eqref{eq:Tc_explicit_ell}, and~\eqref{eq:Tn_final} with $\gamma_1:=c_1+c_5$ and $\gamma_2:=c_6$.
\end{proof}

\Cref{theorem:Error_perturbed_case} provides an explicit second-order comparison between PCG with $U^\theta$ and deflated CG. After one additional PCG iteration, the squared energy-norm error is bounded by the squared deflated CG error multiplied by $\left(\alpha(t, \theta)/\theta\right)^2$, plus a remainder of order $t^2$. The quantity $\alpha(t, \theta)$ is the largest distance from $\theta$ to the nonzero eigenvalues of $\widehat A(t)$; by~\Cref{lm:Btilde}, these eigenvalues agree with the corresponding eigenvalues of $A(t)$ up to $O(t^2)$. Thus, if $\alpha(t, \theta)/\theta<1$, the leading term is contractive.

The explicit factor $\gamma_1q_k^{2\ell}+\gamma_2\ell q_k^{2\ell-1}$ describes how the perturbation remainder depends on the iteration index. When $q_k>1$, this upper bound can grow rapidly with $\ell$ and can be pessimistic; when $0<q_k<1$, its behavior is different and it may decrease as $\ell$ grows. This concerns only the quality of the perturbation estimate and does not imply that either iterative method has increasing error. Increasing $k$ removes more eigenvalues from the nondeflated spectrum and may reduce $q_k$.

Following~\citet{diouane2026spectral}, the next proposition introduces two
scaling choices and shows that both yield a contractive leading factor for
sufficiently small perturbations.

\begin{proposition}\label{prop:contractive_scalings}
Suppose the assumptions of \Cref{theorem:Error_perturbed_case} hold.
\begin{enumerate}
  \item[\textup{(i)}] Then, for sufficiently small $t \geq 0$, 
  the bound \eqref{eq:upperbound_explicit_ell} holds with $\theta = \lambda_k$ and its 
  leading factor is contractive, i.e., $\alpha(t,\lambda_k)/\lambda_k < 1$.
  \item[\textup{(ii)}] Suppose, in addition, that $\lambda_n$ is known, and let 
  $\theta_m := (\lambda_k + \lambda_n)/2$ with 
  $\theta_m \notin \{\lambda_{k+1}, \ldots, \lambda_n\}$. Then, for sufficiently small 
  $t \geq 0$, the bound \eqref{eq:upperbound_explicit_ell} also holds with 
  $\theta = \theta_m$, and its leading factor is strictly smaller than the one 
  in~\textup{(i)}, i.e., $\alpha(t,\theta_m)/\theta_m < \alpha(t,\lambda_k)/\lambda_k < 1$.
\end{enumerate}
\end{proposition}

\begin{proof}
Indeed, since $\lambda_k-\lambda_i\geq\lambda_k-\lambda_{k+1}\geq\gamma_0$
for every $i>k$, we have $\lambda_k\in\Theta_{\gamma_0}$. Let
$t^\ast>0$ be the threshold supplied by
\Cref{theorem:Error_perturbed_case} when it is applied with
$\delta=\gamma_0$, i.e., for any $t \leq t^\ast$, 
  the bound \eqref{eq:upperbound_explicit_ell} holds with $\theta = \lambda_k$. 
%  Thus,\eqref{eq:upperbound_explicit_ell} holds for every $t\in[0,t_r^\ast]$.
At $t=0$, the nonzero eigenvalues of $\widehat A(0)$ are
$\lambda_{k+1},\ldots,\lambda_n$, and hence $\alpha(0,\lambda_k)/\lambda_k
=(\lambda_k-\lambda_n)/\lambda_k<1$.
Since the eigenvalues of $\widehat A(t)$ depend continuously on $t$, and hence so
does $t\mapsto\alpha(t,\lambda_k)$. Therefore, there exists
$\widebar t>0$ such that
$\alpha(t,\lambda_k)/\lambda_k<1$ for every $t\in[0,\widebar t]$.
Hence, by taking $\tau:=\min\{t^\ast,\widebar t\}$
proves~\textup{(i)}.

To prove~\textup{(ii)}, suppose that $\lambda_n$ is known and that $\theta_m\notin\{\lambda_{k+1},\ldots,\lambda_n\}$. Set $\delta
:=\min\left\{
\gamma_0,
\min_{i=k+1,\ldots,n}|\theta_m-\lambda_i|
\right\}$. Then, $0<\delta\leq\gamma_0$ and
$\theta_m\in\Theta_\delta$. Let $t_m^\ast>0$ be the threshold
supplied by~\Cref{theorem:Error_perturbed_case} for this $\delta$, i.e., for any $t \leq t_m^\ast$, 
  the bound \eqref{eq:upperbound_explicit_ell} holds with $\theta = \theta_m$. At $t=0$,
\begin{equation}\label{eq:teta_mleqteta_r}
\frac{\alpha(0,\theta_m)}{\theta_m}
=\frac{\lambda_k-\lambda_n}{\lambda_k+\lambda_n}
<
\frac{\lambda_k-\lambda_n}{\lambda_k}
=\frac{\alpha(0,\lambda_k)}{\lambda_k}
<1.
\end{equation}
By continuity, there exists $\widebar t_m>0$ such that the strict
inequalities in~\eqref{eq:teta_mleqteta_r} remain valid for every
$t\in[0,\widebar t_m]$. Taking $\tau_m:=\min\{\tau,t_m^\ast,\widebar t_m\}$
proves~\textup{(ii)}.
\end{proof}

We note that the choice $\theta= \theta_m$ is
particularly practical in applications such as variational data assimilation,
where $\lambda_n$ is often known a priori.

\section{Numerical experiments}
\label{sec:numerical_experiments}
\subsection{Computational complexity: PCG versus deflated CG}
\label{subsec:computational_cost}
Assume that $S_k=[s_1,\ldots,s_k]$ and $\Lambda_k=\diag(\lambda_1,\ldots,\lambda_k)$ are available. Both deflated CG and PCG with $U^\theta$ use CG as an iterative solver; the main computational difference lies in the cost of applying the deflation projector versus applying the scaled spectral preconditioner.
 
Deflated CG requires the projector $P(t)=I_n -V_k(t)\left(S_k^\top V_k(t)\right)^{-1}S_k^\top$, where $V_k(t)=A(t)S_k$. Constructing this projector requires $k$ products with $A(t)$ to form $V_k(t)$, $\mathcal{O}(nk^2)$ operations to assemble $S_k^\top V_k(t)$, and $\mathcal{O}(k^3)$ operations to factorize it. Once these quantities are available, applying $P(t)$ requires products with $S_k^\top$ and $V_k(t)$ and the solution of a $k\times k$ linear system. Deflated CG stores $S_k$, $V_k(t)$, and the dense matrix $S_k^\top V_k(t)$. In contrast, PCG with the scaled spectral preconditioner $U^\theta$ stores only $S_k$, the associated eigenvalues, and the scalar parameter $\theta$. Applying $U^\theta$ requires one multiplication by each of $S_k^\top$ and $S_k$, together with a diagonal scaling. 

\subsection{Synthetic data experiments}
\label{subsec:data-exp}
We consider the linear system $A(t)x=(A+tE)x=b$, where $A\in\mathbb{R}^{n\times n}$ is diagonal with entries
{\small
\begin{equation}
\label{eq:strakos_eigenvalues}
    \lambda_i = \lambda_n + \frac{n-i}{n-1}
    \left[\psi+(\lambda_1-\lambda_n-\psi)\rho^{i-1}\right],
    \qquad i=1,\ldots,n,
\end{equation}}
where $0<\rho<1$ controls the exponential decay of the spectrum and $0<\psi<\lambda_1-\lambda_n$ ensures that the smallest consecutive eigenvalue gap satisfies $\gamma_0\geq\psi/(n-1)>0$ in exact arithmetic. Setting $\psi=0$ recovers the spectrum used in~\citep{StrakosGreenbaum}. In what follows, we fix $n=4\times10^3$, $\lambda_1=10^6$, $\lambda_n=1$, and $\psi=1$. Throughout this section, $E$ is a fixed randomly generated symmetric matrix satisfying $\|E\|_F=1$, and $b=[1,\ldots,1]^\top/\sqrt{n}$. The initial guesses satisfy~\Cref{ass:z_0,ass:y_0}.

The numerical study is organized in two parts. First, we compare the PCG error with the leading term in~\Cref{theorem:Error_perturbed_case} and discuss the effect of the omitted remainder. Second, we compare the energy-norm convergence of CG, PCG using $U^\theta$, and deflated CG for the same family of test matrices. All methods are run with full reorthogonalization.  
In the supplementary material, we present additional numerical results for a variational data assimilation problem.

\subsubsection{Comparison with the theoretical bound}
\label{sec:numerical_bound_comparison}

We assess the leading term in the upper bound of~\Cref{theorem:Error_perturbed_case} by evaluating {\small$\left\| x^\ast(t) - x_{\ell+1}^{[P]}(t) \right\|_{A(t)}^2$}
 and {\small$(\alpha(t, \theta)^2/{\theta^2})
 \left\|x^\ast(t) - x_{\ell}^{[\text{D}]}(t)\right\|_{A(t)}^2$}, as functions of the perturbation size $t$ for different $\ell$ and $\rho$. We set $k=60$ and choose $\theta_m$ as scaling. For this choice, $\alpha(t,\theta_m)/\theta_m\approx0.99<1$ over the tested range of $t$. \Cref{tab:diagnostic-diagonal-case} reports, for $\rho\in\{0.9,0.94\}$, the gap $\gamma_0$, the ratio $\lambda_k/\lambda_n$ governing
$q_k$, and the separation of $\theta_m$ from the nondeflated spectrum.

\begin{table}[htbp]
\centering
\begin{tabular}{cccc}
\hline
 & $\gamma_0$ & $\lambda_k/\lambda_n$ & $\min_{i=k+1,\ldots,n}|\theta_m-\lambda_i|$\\
\hline
$\rho=0.9$  & $2.5\times10^{-4}$ & $1.969\times10^{3}$ & $4.387\times10^{1}$  \\
$\rho=0.94$ & $2.5\times10^{-4}$ & $2.559\times10^{4}$ & $1.254\times10^{2}$ \\
\hline
\end{tabular}
\caption{Diagnostic quantities: the minimum eigenvalue gap $\gamma_0$, the ratio $\lambda_k/\lambda_n$ governing $q_k$, and the separation $\min_{i>k}|\theta_m-\lambda_i|$ of the midrange scaling parameter from the nondeflated spectrum.}
\label{tab:diagnostic-diagonal-case}
\end{table}

From~\Cref{tab:diagnostic-diagonal-case}, \Cref{ass:simple_spectrum} is satisfied; choosing $\delta=\gamma_0$ ensures that $\theta_m\in\Theta_\delta$.
The behavior of the omitted remainder in~\Cref{theorem:Error_perturbed_case} is governed by $q_k$, whose size is controlled mainly by the ratio $\lambda_k/\lambda_n$. From~\Cref{tab:diagnostic-diagonal-case}, the ratio $\lambda_k/\lambda_n>10^3$ for both cases; consequently, the remainder $\gamma_1q_k^{2\ell}+\gamma_2\ell q_k^{2\ell-1}$ becomes large for large values of $\ell$.

\begin{figure}[htb]
    \centering
    \includegraphics[width=1.\linewidth]{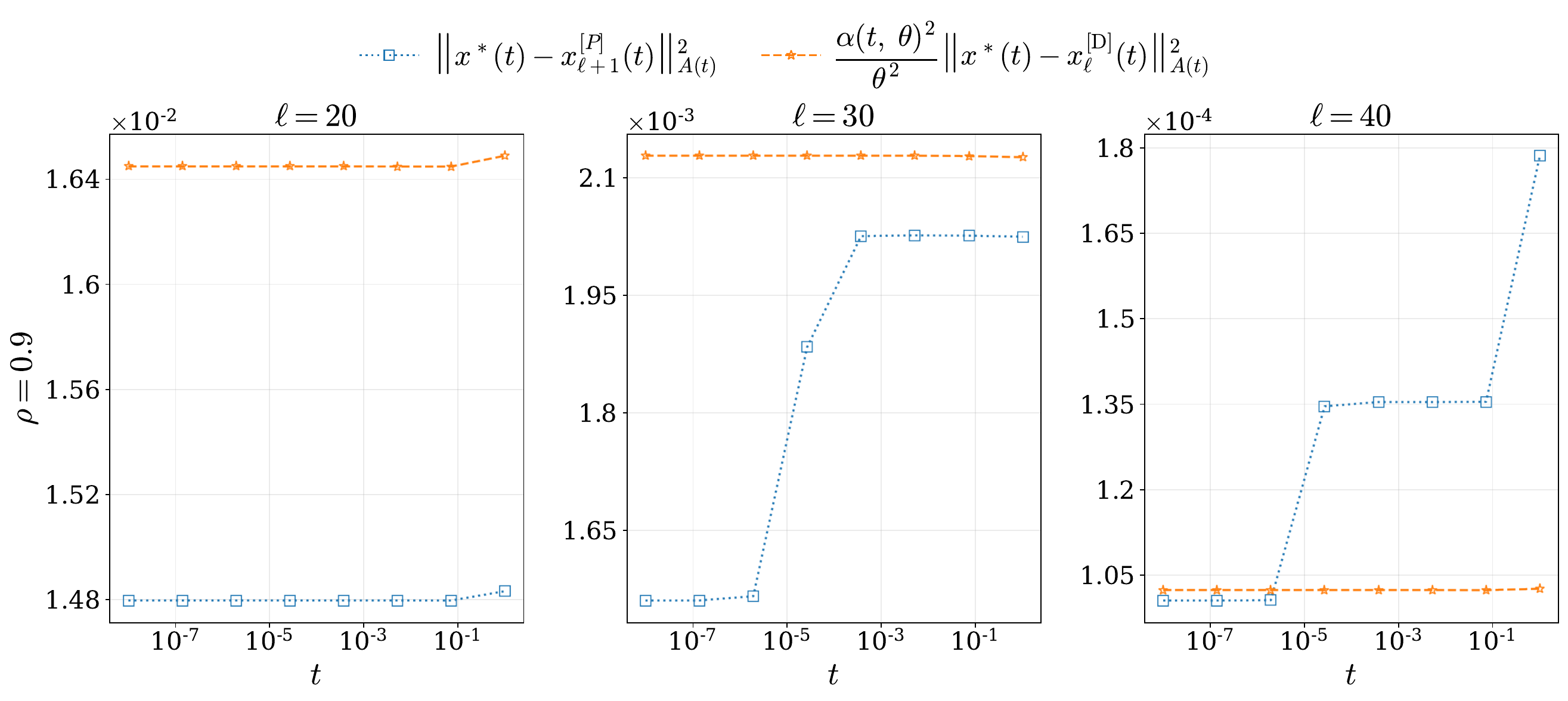}
    \caption{The squared PCG error (blue) and the leading term of the bound in~\Cref{theorem:Error_perturbed_case} (orange) for $\rho = 0.9$ and $\ell\in\{20,30,40\}$.}
    \label{fig:bound_0.9}
\end{figure}

\begin{figure}[htb]
    \centering
    \includegraphics[width=1.\linewidth]{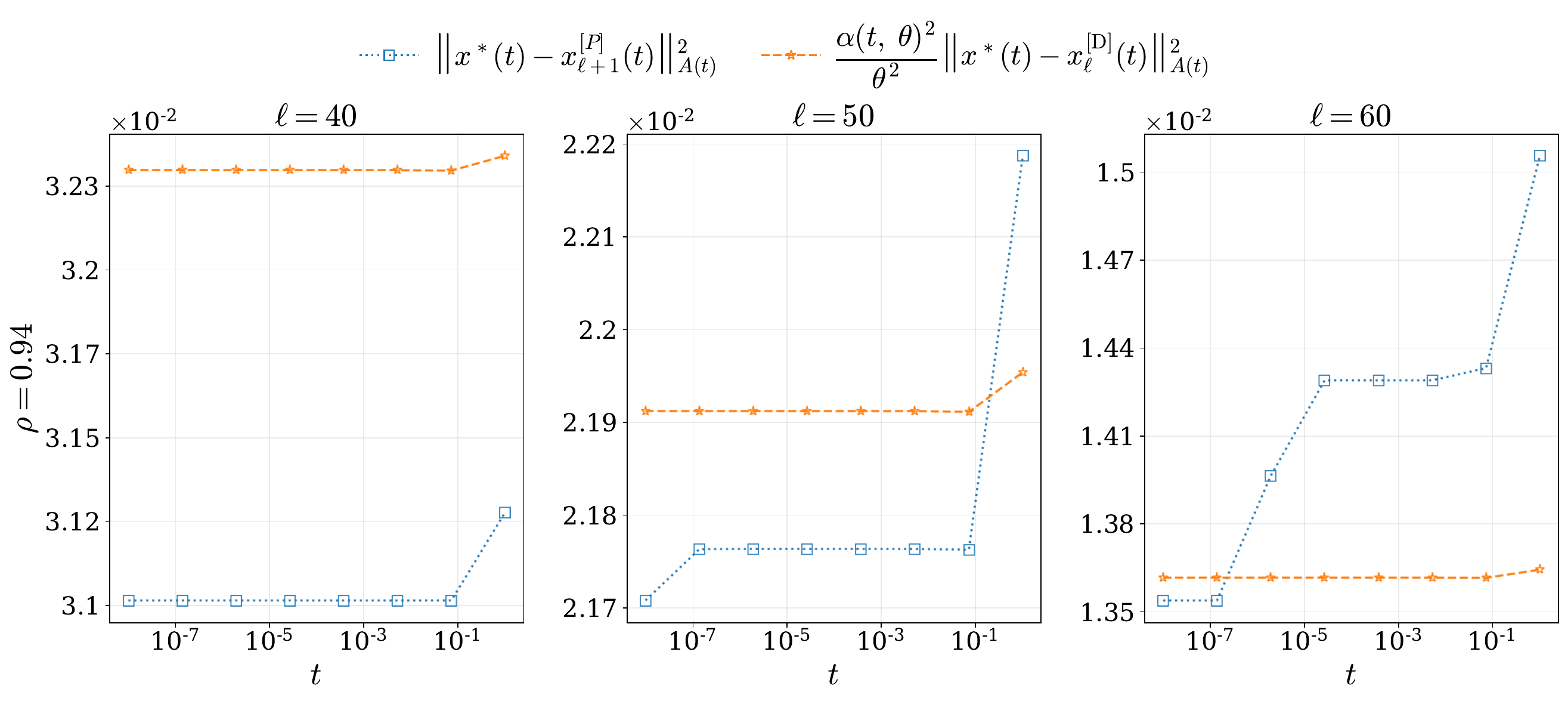}
    \caption{The squared PCG error (blue) and the leading term of the bound in~\Cref{theorem:Error_perturbed_case} (orange) for $\rho=0.94$ and $\ell\in\{40,50,60\}$.}
    \label{fig:bound_0.94}
\end{figure}

For $\rho=0.9$ and $\ell=20,30$, the leading term of the right hand side of~\eqref{eq:upperbound_explicit_ell} lies above the PCG error over the full range of perturbations. For $\ell=40$, it falls below the PCG error, although the discrepancy decreases for small $t$.
For $\rho=0.94$, the leading term lies above the PCG error for $\ell=40$ over the full range of perturbations. For $\ell=50$, it falls below the PCG error for $t$ close to $1$, and for $\ell=60$ once $t$ exceeds $10^{-6}$.
This behavior does not contradict~\Cref{theorem:Error_perturbed_case}: the theorem bounds the PCG error by the leading term \emph{plus} an $O(t^2)$ remainder, which can be large here because of $q_k$. Its influence becomes apparent at later iterations, when the energy-norm errors are small. Thus, as predicted by~\Cref{theorem:Error_perturbed_case}, the accuracy of the leading-order approximation depends strongly on $\ell$.

\subsubsection{Comparison of energy-norm errors}
\label{Sec:NumericalResultsC0}
These numerical experiments complement the preceding comparison of the theoretical bound by examining the convergence behavior of the different solvers. We consider $k\in\{60,80\}$, choose $\theta\in\{\theta_m,\lambda_k\}$, and apply the strategies listed in the table below for $t\in\{1,10^{-3},10^{-6}\}$. \Cref{tab:methods} reports the methods used in the numerical experiments.
\begin{table}[ht]
\centering
\begin{tabular}{ll}
\hline
Method & Description \\ \hline
\textbf{CG} & CG applied to $A(t) x = b$, \quad starting with $x_0 = z_0$\\  
\textbf{Deflated CG} &  CG applied to the projected system $P(t) A(t) z = P(t)b$\\
$\mathbf{PCG}(U^{\theta_\mathrm{m}})$ & CG applied to $U^{\theta} A(t) U^{\theta}y = U^{\theta} b$, with $\theta=\theta_m$ and $x=U^\theta y$\\
$\mathbf{PCG}(U^{\lambda_k})$ & CG applied to $U^{\theta} A(t) U^{\theta}y = U^{\theta} b$, with $\theta=\lambda_k$ and $x=U^\theta y$\\\hline
\end{tabular}
\caption{Description of methods used in the numerical experiments}
\label{tab:methods}
\end{table}

\begin{figure}[htb]
    \centering
    \includegraphics[width=1.\linewidth]{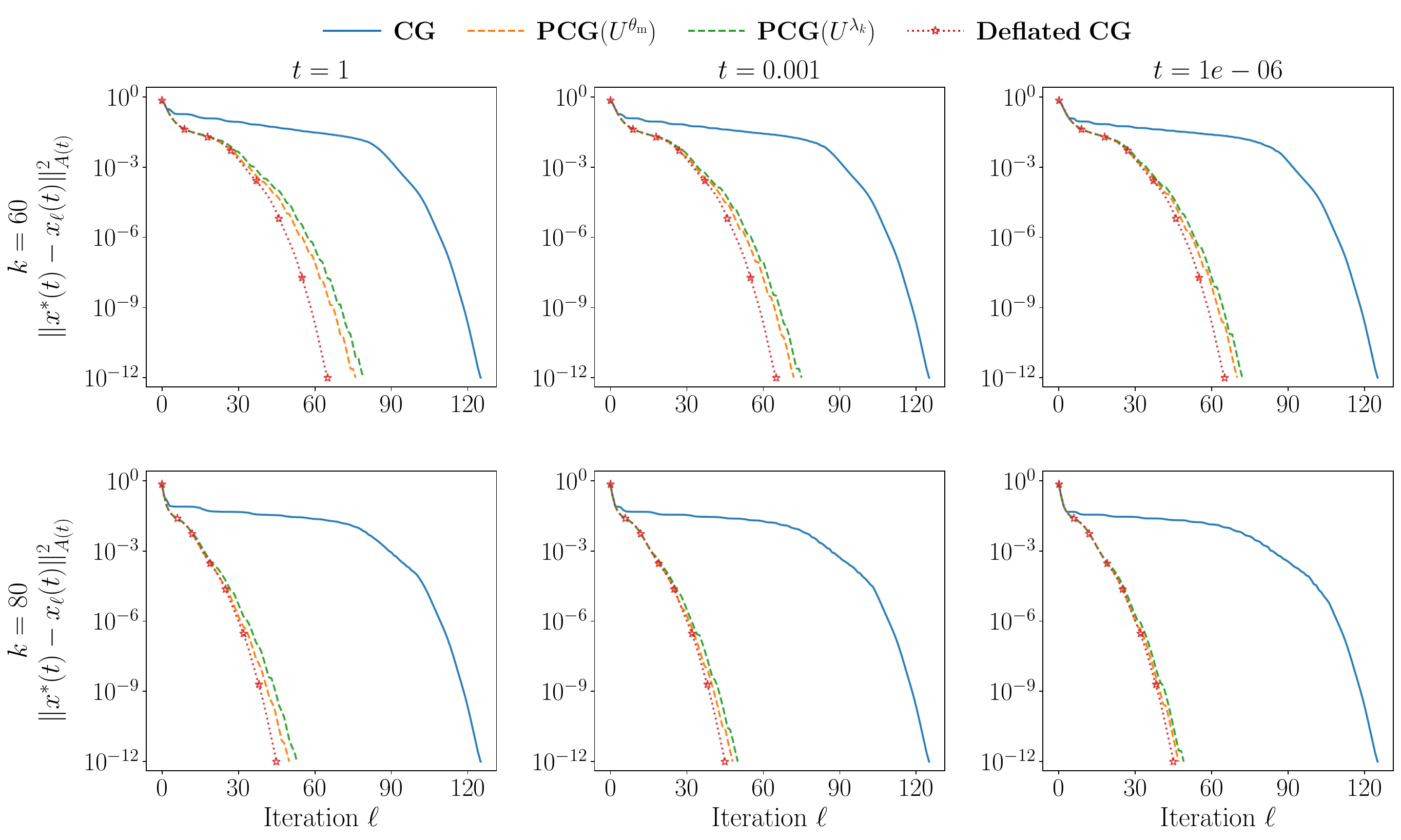}
    \caption{Squared energy-norm convergence for $\rho=0.9$. The top and bottom rows correspond to $k=60$ and $k=80$, respectively; from left to right, the columns correspond to $t=1$, $10^{-3}$, and $10^{-6}$.}
    \label{fig:energy_rho09}
\end{figure}

\begin{figure}[htb]
    \centering
    \includegraphics[width=1.\linewidth]{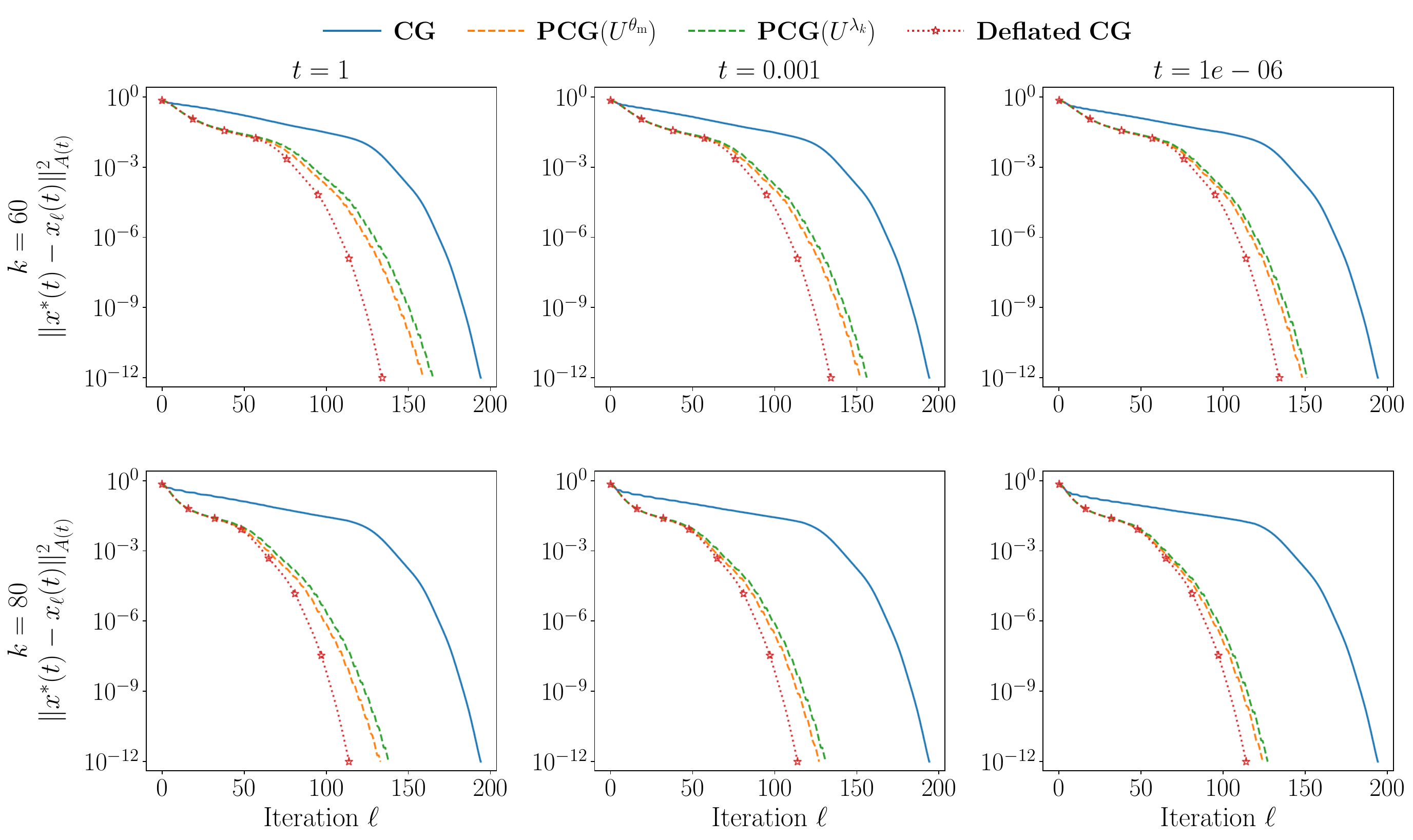}
    \caption{Same as~\Cref{fig:energy_rho09} for $\rho = 0.94$.}
    \label{fig:energy_rho094}
\end{figure}

During the first iterations, the perturbation size $t$ has little effect on the difference between deflated CG and $\mathbf{PCG}(U^{\theta_m})$. Their convergence curves separate mainly near the end of the iteration history, once the error is small. For both $\rho=0.9$ and $\rho=0.94$, $\mathbf{PCG}(U^{\theta_m})$ substantially improves upon CG and closely follows deflated CG at moderate accuracy. The difference is most visible at high accuracy and for the largest perturbation, $t=1$; it narrows for $t=10^{-3}$ and $10^{-6}$.
Increasing $k$ from $60$ to $80$ reduces the gap between $\mathbf{PCG}(U^{\theta_m})$ and deflated CG, especially during the final iterations. Increasing $k$ removes more of the largest eigenvalues and decreases $\lambda_k$. Consequently, the terms $q_k^{2\ell}$ and $\ell q_k^{2\ell-1}$ in the perturbation remainder are reduced.

Thus, $t$ primarily affects the number of late iterations, rather than the initial convergence rate. PCG with $\theta_m$ retains most of the benefit of deflation but may require a few additional iterations at high accuracy. This cost can be modest relative to constructing the deflation projector, which requires $k$ matrix--vector products with $A(t)$. The preconditioned approach is therefore especially attractive when these products are expensive.
The same conclusion holds for $\mathbf{PCG}(U^{\lambda_k})$, although $\mathbf{PCG}(U^{\theta_m})$ performs better, consistently with~\Cref{prop:contractive_scalings}.

\section{Conclusion}
\label{sec:conclusion}
We analyzed sequences of SPD linear systems in which spectral information from a reference matrix $A$ with simple eigenvalues is reused for a perturbed matrix $A(t)=A+tE$. Our comparison considered deflated CG and PCG equipped with a scaled spectral preconditioner constructed from the same selected eigenvectors of $A$.

An energy-norm characterization of deflated CG, combined with first-order perturbation analysis, yields an asymptotic comparison of the two methods. After one additional PCG iteration, the squared PCG error is bounded, up to higher-order perturbation terms, by an explicit multiple of the squared deflated CG error. The analysis also motivates the scaling parameter $\theta_m$.
Numerical experiments show that the scaled spectral preconditioner closely reproduces the convergence of deflated CG while avoiding the construction of the deflation projector and its associated products with the perturbed matrix. PCG with this preconditioner is therefore a promising alternative to deflation for sequences of closely related large-scale systems.

An important direction for future work is to extend the perturbation analysis to SPD matrices with multiple eigenvalues. We also fix the perturbation direction $E$ and vary only its magnitude $t$; allowing $E$ to vary and deriving bounds over a class of symmetric perturbations would yield a more general analysis.

\small
\bibliographystyle{abbrvnat}  
\bibliography{abbrv,references}

\clearpage

%\tableofcontents

\end{document}

%% file: ex_shared.tex
\headers{A deflation-based preconditioner for PCG}{Y. Diouane, S. G\"urol, O. Mouhtal and D. Orban}

\title{Deflation-Based Spectral Preconditioning for the Conjugate Gradient Method on Perturbed  Systems \thanks{Submitted to the editors DATE.
\funding{This research was partially funded by the Natural Sciences and Engineering Research Council of Canada (NSERC) Discovery Grants held by Diouane and Orban. G\"urol and Mouhtal  were partially funded by the French national programme LEFE/INSU.}}}

\author{Y. DIOUANE\thanks{GERAD and Department of Mathematics and Industrial Engineering, Polytechnique Montr\'eal. 
(\email{youssef.diouane@polymtl.ca},
\email{dominique.orban@polymtl.ca},
\email{oussama-2.mouhtal@polymtl.ca}).}
\and S. G\"urol\thanks{CERFACS / CECI
CNRS UMR 5318, Toulouse, France.
  (\email{gurol@cerfacs.fr}).}
\and  O. MOUHTAL\footnotemark[3] \footnotemark[2]  
\and D. ORBAN\footnotemark[2]  
}